\documentclass{article}

\usepackage[english]{babel}

\usepackage[a4paper,top=2cm,bottom=2cm,left=3cm,right=3cm,marginparwidth=1.75cm]{geometry}

\usepackage[colorlinks=true, allcolors=blue]{hyperref}
\usepackage{amssymb, amscd, amsthm, amsmath,latexsym, amstext,mathrsfs,multicol,CJK, mathtools}
\usepackage[all]{xy}
\usepackage[dvips]{graphicx}
\usepackage{relsize}

\newtheorem{thm}{Theorem}[section]
\newtheorem{coro}[thm]{Corollary}
\newtheorem{lemma}[thm]{Lemma}
\newtheorem{prop}[thm]{Proposition}

\theoremstyle{definition}
\newtheorem{defn}[thm]{Definition}
\newtheorem{example}[thm]{Example}

\newtheorem{remark}[thm]{Remark}
\newtheorem{ques}[thm]{Question}

\numberwithin{equation}{thm}
\numberwithin{equation}{subsection}

\newcommand{\rank}{\mathrm{rank}}
\newcommand{\Spm}{\mathrm{Spm}}
\newcommand{\car}{\mathrm{char}}

\newcommand{\cls}{\mathrm{cls}}
\newcommand{\gen}{\mathrm{gen}}
\newcommand{\spn}{\mathrm{spn}}
\newcommand{\rO}{\mathrm{O}}
\newcommand{\fp}{\mathfrak{p}}
\newcommand{\fq}{\mathfrak{q}}

\newcommand{\bfO}{\mathbf{O}}
\newcommand{\bfA}{\mathbf{A}}
\newcommand{\bfI}{\mathbf{I}}

\title{Integral Springer Theorem for Quadratic Lattices \\ under Base Change of Odd Degree}

\author{Yong HU, Jing LIU, and Fei XU}
\date{}

\begin{document}
	\maketitle

	\begin{abstract}
		 A quadratic lattice $M$ over a Dedekind domain $R$ with fraction field $F$ is defined to be a finitely generated torsion-free $R$-module equipped with a non-degenerate quadratic form on the $F$-vector space $F\otimes_{R}M$. Assuming that $F\otimes_{R}M$ is isotropic of dimension $\geq 3$ and that $2$ is invertible in $R$, we prove that a quadratic lattice $N$ can be embedded into a quadratic lattice $M$ over $R$ if and only if $S\otimes_{R} N$ can be embedded into $S\otimes_{R}M$ over $S$, where $S$ is the integral closure of $R$ in a finite extension of odd degree of $F$. As a key step in the proof, we establish several versions of the norm principle for integral spinor norms, which may be of independent interest.
	\end{abstract}

	%\tableofcontents
	
	\section{Introduction}

Let $F$ be a field of characteristic $\neq 2$. A well known theorem  asserts that an anisotropic quadratic form over $F$ remains anisotropic over any finite extension of odd degree of $F$. This was conjectured by Witt in \cite{Wit37} and was proved by Springer in \cite{Spr52}. By using Witt cancellation, one can show that this classical theorem of Springer is essentially equivalent to a relative version:
For two non-degenerate quadratic spaces $(V,f)$ and $(W,g)$ over $F$, $W$ is represented by $V$ over $F$ (in the sense that there is an $F$-linear map $W\to V$ compatible with the quadratic forms $f$ and $g$)  if and only if $E\otimes_FW$ is represented by $E\otimes_FV$ for some finite extension $E/F$ of odd degree.
	
For integral quadratic forms over number fields, Earnest and Hsia studied Springer-type theorems for spinor genera in \cite{EH77} and \cite{EH78}. Inspired by their work,  the third-named author conjectured an arithmetic analogue of the relative version of Springer's theorem for indefinite quadratic forms over the ring of integers of a number field in  \cite{X99}, where among others a proof for forms over $\mathbb Z$ was given. Recently, this conjecture has been solved completely by Z. He in \cite{H23}. The analogous statement for definite integral quadratic forms can fail, as shown by Daans, Kala, Kr\'asensk\'y and Yatsyna in \cite{DaansKalaKrasenskyYatsyna24}.
	
On the other hand, the classical version of Springer's theorem stated in terms of isotropy of forms has been extended to unimodular quadratic forms over semilocal rings by Panin and Rehmann in \cite{PR07}, Panin and Pimenov in \cite{PP10}, and Gille and Neher in \cite{GN21}.

\medskip

In this paper, we consider quadratic forms over general Dedekind domains and prove a relative Springer theorem for them under some mild assumptions. Throughout this paper we work with a field $F$ of characteristic $\car(F)\neq 2$ and let $R$ be a Dedekind domain with fraction field $F$. As in \cite{O'M00}, a (non-degenerate) \emph{quadratic lattice} over $R$ is defined to be a finitely generated torsion-free $R$-module $M$ together with a non-degenerate quadratic form $Q: F\otimes_RM\to F$ defined on the $F$-vector space $F\otimes_{R}M$. When there is no risk of confusion, we will simply say that $M$ is a quadratic lattice over $R$.

For two quadratic lattices $(M,Q)$ and $(N,q)$ over $R$, a \emph{representation} of $N$ in $M$ is an $R$-module homomorphism $\sigma: N\to M$ such that $q=Q\circ\sigma$. By the non-degeneracy of $q$, such a homomorphism $\sigma$ is always injective, so we also say that $\sigma$ is an \emph{embedding} of $N$ into $M$. When such a $\sigma$ exists, we say that $N$ is \emph{represented by} $M$ or that $N$ can be \emph{embedded} into $M$ over $R$.

If $S$ is the integral closure of $R$ in a finite extension of $F$, then $S$ is again a Dedekind domain, by the Krull--Akizuki theorem. Thus, for any quadratic lattice $M$ over $R$, the base extension $S\otimes_RM$ has a natural structure of quadratic lattice over $S$.

\begin{ques}\label{Spring-question}
 Let $M$ and $N$ be quadratic lattices over $R$. Let $S$ be the integral closure of $R$ in a finite extension $E/F$ of odd degree.
Suppose that $S\otimes_{R}N$ can be embedded into $S\otimes_{R}M$ over $S$.
Can we embed $N$ into $M$ over $R$\,?
\end{ques}

The main result of this paper is the following theorem.
	
\begin{thm} \label{intro}
Assume that the quadratic space $F\otimes_{R}M$ is isotropic of dimension $\geq 3$ and that $2$ is invertible in $R$.

Then Question \ref{Spring-question} has a positive answer.
\end{thm}

The proof of this theorem will be given in the end of \S\;\ref{sec5.2}.

\begin{remark}
It is worth giving some remarks about our assumptions in Theorem \ref{intro}:
	
(1) Neither the assumption that $\dim(F\otimes_RM)\geq 3$ nor that $F\otimes_{R}M$ is isotropic can be removed in general. Explicit counterexamples are provided in Examples \ref{c-exam} and \ref{isotropy-exam}. These assumptions are needed in our proof to ensure that the spin group of the quadratic space satisfies strong approximation over $R$ (Remark \ref{rmkSA}).

Note that in the conjecture raised by the third-named author in \cite[p.175]{X99} (which is now proved in \cite{H23}), where $R$ is the ring of integers of a number field, one only needs to assume that $F\otimes_{R}M$ is isotropic over an archimedean completion. See, however, Example \ref{NumberField-exam} and Remark \ref{definite-exam}.
	
(2) Our proof uses the assumption that $2$ is invertible in $R$. It is likely that the theorem remains valid without this assumption.
\end{remark}

The basic strategy of our proof of Theorem \ref{intro} is similar to that of the proof in the number field case in \cite{X99} and \cite{H23}. We first prove a local integral Springer theorem. To pass from the local case to the global case, we use the adelic language for orthogonal groups and the genus theory for quadratic lattices to endow the set of isomorphism classes of lattices in which $N$ embeds with a natural group structure. Then a group homomorphism between certain quotients of id\`ele groups is associated to Question \ref{Spring-question}, and so the problem reduces to the injectivity of this group homomorphism. To prove the desired injectivity, we have to prove several versions of norm principles for integral spinor norms. While what we use ultimately is an adelic version, proving the norm principle in the local case turns out to be the core of the most technical arguments in the whole paper.

Actually, by using the notion of spinor genus (Definition \ref{defn4.1}), we obtain a refined version of Theorem \ref{intro}, which can be thought of as a Springer theorem for the embedability into a spinor genus (Theorem \ref{spin-Springer}).

\medskip

Compared to the classical number field case, one of the main difficulties in our situation is that there is no control of dimensions of anisotropic lattices over the local completions, since the residue fields of $R$ can be arbitrary so that anisotropic quadratic forms can exist over the residue fields in arbitrarily large dimensions. This is partially responsible for the lengthy arguments in \S\S\;\ref{sec3.2} and \ref{sec3.3}.

As we have said above, a crucial step in our proof consists in establishing an appropriate norm principle for integral spinor norms. This phenomenon resembles an integral version of a special case of the norm principles for reductive groups studied previously in \cite{M95},  \cite{G97} and \cite{BM02}. So we think that our results in this direction (cf. \S\S\;\ref{sec3.3} and \ref{norm}) may be of independent interest.

\medskip

The paper is organized as follows. We first show  in \S\;\ref{loc} that the answer to Question \ref{Spring-question} is affirmative when $R$ is a complete discrete valuation ring containing $1/2$. This establishes in particular a local version of Theorem \ref{intro}.  In \S\;\ref{sec3},  we prove the most technical results that are needed in the proof of the main theorem, including reduction formulas for sets of transporters (see below for definition) between two lattices and norm principles for local integral spinor norms. We review the adelic language and the genus theory in \S\;\ref{sec4.1}, and we use the notions of genus and spinor genus to interpret lattice classes and spinor genera in which a given lattice can be embedded as elements of certain abelian 2-torsion groups in \S\;\ref{sec4.2}. Based on the local results obtained in \S\S\;\ref{sec3} and \ref{sec4}, we prove an adelic norm principle in \S\;\ref{norm}, and then, in \S\;\ref{sec5.2}, we use it to prove a Springer-type theorem in terms of spinor genera and deduce Theorem \ref{intro} by using strong approximation for spin groups. Finally, in \S\;\ref{sec5.3} we give some examples to explain that the assumptions in Theorem \ref{intro} cannot be dropped in general.

\bigskip

\noindent {\bf Notation and terminology.} As we have mentioned before, $F$ denotes a field of characteristic $\neq2$ and $R$ denotes a Dedekind domain with fraction field $F$. Notation and terminology for quadratic spaces and quadratic lattices are standard if not explained, and generally follow those used in O'Meara's book \cite{O'M00}.
Unless otherwise stated, quadratic spaces and lattices under our consideration are always assumed to be non-degenerate.

For a quadratic space $(V,Q)$ over $F$, let  $\langle\cdot,\cdot\rangle:  V\times V\to F$ denote the symmetric bilinear form associated to the quadratic form $Q$, so that $\langle x,x\rangle=Q(x)$ for all $x\in V$.  Let $\mathrm{O}(V)$ and $\mathrm{O}^+(V)$ denote the orthogonal group and the special orthogonal group of $(V,Q)$ respectively. For any vector $u\in V$ with $Q(u)\neq 0$, the \emph{reflection} along $u$ is the map $\tau_u: V \to V$ given by
\[\tau_u(x)=x-\frac{2 \langle x, u \rangle }{Q(u)}u, \ \ \text{for all }  x\in V.
 \]We have $\tau_u\in \mathrm{O}(V)$. The well known Cartan--Dieudonn\'e theorem (cf. \cite[43:3. Theorem]{O'M00}) asserts that the group $\mathrm{O}(V)$ is generated by reflections (i.e., elements of the form $\tau_u$).

If $W$ is a (non-degenerate) subspace of $V$, one has $V=W\perp W^{\perp}$. Then  $\mathrm{O}^+(W)$ and $\mathrm{O}(W)$  can be viewed as subgroups of $\mathrm{O}^+(V)$ and $\mathrm{O}(V)$ respectively, via the identifications
$$\mathrm{O}^+(W) = \{\sigma\in \mathrm{O}^+(V): \sigma|_{W^{\perp}}=\mathrm{id}\} \ \ \ \text{and} \ \ \  \mathrm{O}(W) = \{\sigma\in \mathrm{O}(V): \sigma|_{W^{\perp}}=\mathrm{id}\}. $$

 For a quadratic $R$-lattice $M$ with $V=F\otimes_RM$, its orthogonal group $\mathrm{O}(M)$ and special orthogonal group $\mathrm{O}^+(M)$ are defined by
\[
\mathrm{O}(M)=\{\sigma\in\mathrm{O}(V)\,|\,\sigma(M)=M\} \quad\text{and}\quad \mathrm{O}^+(M)=\mathrm{O}(M)\cap\mathrm{O}^+(V).
\]If $N$ is a quadratic lattice in $V=F\otimes_RM$, we define
\[
 X(M/N)= \{ \sigma \in \mathrm{O}(V)\,|\, N \subseteq \sigma (M) \} \quad \text{and} \quad X^+(M/N)= \{ \sigma \in \mathrm{O}^+(V)\,|\, N \subseteq \sigma (M) \}.
\]For convenience, let us call elements of $X(M/N)$ \emph{transporters} of $N$ in $M$ and call elements of $X^+(M/N)$ \emph{proper transporters} of $N$ in $M$.

The \emph{scale} of a quadratic $R$-lattice $M$ is denoted by $\mathfrak{s}(M)$. It is the fractional ideal of $R$ generated by the elements $\langle x, y  \rangle$ for all $x,y\in M$. For a fractional ideal $\mathfrak{a}$ of $R$, the $R$-module $\mathfrak{a} \otimes_R M$ equipped with the quadratic form of the space $F\otimes_RM$ is also a quadratic lattice over $R$. When $\mathfrak{a}=\alpha R$ is principal, we simply write $\alpha M$ for $\mathfrak{a} \otimes_R M$.

We say that $M$ is \emph{modular} if $M=\mathfrak{s}(M)\otimes_RM^{\#}$, where
$$ M^{\#}= \{ x\in F\otimes_RM\,|\,   \langle x, M \rangle \subseteq R \}$$
is the dual lattice of $M$. If $M$ is modular with $\mathfrak{s}(M)=R$, then it is called \emph{unimodular}.

For a commutative ring $A$, its  group of units is denoted by $A^\times$.

\section{Some applications of O'Meara's local embedding criterion}\label{loc}
	
Throughout this section, we assume that $R$ is a complete discrete valuation ring in which $2$ is invertible.

We will use the well known theory of Jordan splittings for quadratic lattices over $R$ as presented in \cite[Chapter IX]{O'M00}.  In particular, any quadratic lattice $L$ has a Jordan splitting, i.e., an orthogonal splitting
$L=L_1\perp\cdots\perp L_t$ where $L_1, \cdots, L_t$ are modular lattices with $\mathfrak{s}(L_1)\supset\cdots\supset\mathfrak{s}(L_t) $.
Moreover, the number $t$ and the ordered sequences
\[
\big(\rank(L_1), \cdots, \rank(L_t)\big)\quad\text{ and }\quad \big(\mathfrak{s}(L_1), \cdots, \mathfrak{s}(L_t)\big)
\] are uniquely determined by $L$. (In fact, the proofs of the above statements given in \cite[\S\S\;91 and 92]{O'M00} are valid without assuming the finiteness of the residue field of $R$.)
	
\subsection{Embedding criterion and integral Springer theorem in the local case}

For a quadratic lattice $L$ with a Jordan splitting $L=L_1\perp\cdots\perp L_t$, we write
$$
L_{\leq i}\coloneqq\mathop{\mathlarger{\mathlarger{\perp}}}\limits_{\mathfrak s(L_r)\supseteq \mathfrak p^i}L_r\quad \text{ for each } i\in \mathbb{Z},
$$ where $\mathfrak p$ is the maximal ideal of $R$. (The notation used in \cite{O'M58} is $\mathfrak{L}_i$.)

\begin{thm}[O'Meara]\label{localrep}
Let $\ell$ and $L$ be quadratic lattices over $R$ with given Jordan splittings
$$\ell=\ell_1\perp\cdots\perp\ell_s \ \ \ \text{ and } \ \ \ L=L_1\perp\cdots\perp L_t . $$

Then $\ell$ can be embedded into $L$ over $R$ if and only if $F\otimes_R \ell_{\leq i}$ can be embedded into $F\otimes_R L_{\leq i}$ as quadratic spaces over $F$ for all $i\in \mathbb Z$.
\end{thm}
\begin{proof}
See {\cite[p.850. Thm.1]{O'M58}}. Indeed, O'Meara's paper \cite{O'M58} only considered the case where the residue field of $R$ is finite. However, the proof of this theorem does not depend on this finiteness assumption and is still valid in our case.
\end{proof}

As an immediate consequence of Theorem \ref{localrep}, we obtain the following local integral Springer theorem (where the rank of $M$ can be $2$).
	
\begin{thm}\label{local-springer}
Question \ref{Spring-question} has a positive answer for the complete discrete valuation ring $R$ (where $2$ is invertible). That is, for any quadratic lattices $M$ and $N$ over $R$,  $N$ can be embedded into $M$ over $R$ if and only if  $S\otimes_{R}N$ can be embedded into $S\otimes_{R}N$ over $S$, where $S$ is the integral closure of $R$ in a finite extension $E/F$ of odd degree.
\end{thm} 	
\begin{proof}  Let
$$ M=M_1\perp \cdots\perp M_s \ \ \ \text{and} \ \ \ N=N_1\perp \cdots \perp N_t $$ be Jordan splittings of $M$ and $N$ respectively. Then
$$ S\otimes_R M= (S\otimes_R M_1) \perp \cdots \perp (S\otimes_R M_s) \ \ \ \text{and} \ \ \ S\otimes_R N = (S\otimes_R N_1) \perp \cdots \perp (S\otimes_R N_t) $$
are Jordan splittings of $ S\otimes_R M$ and $ S\otimes_R N $ respectively. By Theorem \ref{localrep}, $S\otimes_RN$ can be embedded into $S\otimes_{R}M$ over $S$ if and only if $E\otimes_S(S\otimes_RN)_{\leq i}$ can be embedded into $E\otimes_S(S\otimes_RM)_{\leq i}$ as quadratic spaces over $E$ for all $i\in \mathbb Z$. By the classical Springer theorem in \cite{Spr52}, the latter condition is equivalent to saying that  $F\otimes_R N_{\leq i}$ can be embedded into $F\otimes_R M_{\leq i}$ for all $i\in \mathbb Z$. So we get the desired result by applying Theorem \ref{localrep} once again.
\end{proof}

\subsection{Uniqueness of Jordan splittings and cancellation law}

The results in this subsection are proved in \cite[92:2a and 92:3]{O'M00} under the extra assumption that $R$ has finite residue field.
In our case the residue field of $R$ can be arbitrary, and our proofs are based on Theorem \ref{localrep}.

\begin{coro}\label{latticeiso}
Let $L$ and $K$ be modular quadratic lattices over $R$ with the same scale. Then $L\cong K$ as quadratic lattices if and only if $F\otimes_R L\cong F\otimes_R K$ as quadratic spaces.
\end{coro}
\begin{proof}
If $F\otimes_RL\cong F\otimes_RK$, then there is an embedding $\sigma: L\to K$ by Theorem \ref{localrep}. Since $\sigma(L)$ is a modular sublattice of $K$ with the same scale, one obtains that $\sigma(L)$ splits $K$ by \cite[82:15a]{O'M00}. Since $L$ and $K$ have the same rank, one concludes that $\sigma(L)=K$, as desired.
\end{proof}

\begin{coro}\label{uniqueness}
For any  quadratic lattice $L$ over $R$, the Jordan splittings of $L$ are unique up to isomorphism.
\end{coro}
\begin{proof}
Let $$L=L_1\perp \cdots \perp L_t \quad\text{and} \quad L=K_1\perp\cdots\perp K_t$$ be two Jordan splittings of $L$. Then $\mathfrak{s}(L_i)=\mathfrak{s}(K_i)$ and $\rank(L_i)=\rank(K_i)$ for $1\leq i\leq t$, by \cite[92:2. Theorem]{O'M00}. By Theorem \ref{localrep},  there are embeddings
 $$F\otimes_RL_1\perp\cdots\perp F\otimes_RL_i\longrightarrow F\otimes_RK_1\perp\cdots\perp F\otimes_RK_i$$ for $1\leq i\leq t$. Comparing the dimensions, we obtain $F\otimes_RL_1\perp\cdots\perp F\otimes_RL_i\cong F\otimes_RK_1\perp\cdots\perp F\otimes_RK_i$ for each $i$. By the Witt cancellation theorem for quadratic spaces (\cite[42:16. Theorem]{O'M00}), we have $F\otimes_RL_i\cong F\otimes_RK_i$ for each $i$. Hence $L_i\cong K_i$ for each $i$, by Corollary \ref{latticeiso}.
\end{proof}

\begin{coro}[Cancellation Law]\label{cancellation}
Suppose given orthogonal splittings of quadratic lattices $L=L'\perp L''$ and $K=K'\perp K''$ over $R$.

If $ L'\cong K' $ and $L\cong K$,  then $L''\cong K''$.
\end{coro}
\begin{proof}
Since a Jordan splitting of $L$ (resp. $K$) can be obtained by grouping together Jordan splittings of $L'$ (resp. $K'$) and $L''$ (resp. $K''$), one only needs to prove the case where $L,L',L'',K,K',K''$ are all modular with the same scale, by Corollary \ref{uniqueness}. Now the assertion follows  immediately from Corollary \ref{latticeiso} and the Witt cancellation theorem for quadratic spaces.
\end{proof}

\section{Local analyses of isometries, transporters and integral spinor norms}\label{sec3}

As in the previous section, let $R$ be a complete discrete valuation ring (with fraction field $F$) and assume that  2 is invertible in $R$.
Let $\pi\in R$ be a uniformizer of $R$, and let $\mathrm{ord}: F\to \mathbb{Z}\cup\{\infty\}$ denote the normalized discrete valuation on $F$.

\subsection{Observations on local integral orthogonal groups}

We say that a binary quadratic $R$-lattice $H$ is  \emph{hyperbolic} if it is generated by two vectors $x,y\in H$ such that $Q(x)=Q(y)=0$ and $\langle x,y\rangle R=\mathfrak{s}(H)$.

\begin{prop} \label{hyper} For any modular quadratic lattice $L$ over $R$, the following statements are equivalent:

\begin{enumerate}
  \item[(1)] $L$ is not split by a  hyperbolic binary lattice.
  \item[(2)] $Q(x) R= \mathfrak s(L)$ for all $x\in L\setminus \pi L$.
  \item[(3)] For any  $x_0\in L$ such that $\mathfrak s(L)=Q(x_0) R$, the quadratic space $(L/\pi L,  \bar { Q } )$ over the residue field $R/\pi R$  is anisotropic, where
$\bar{Q} (\bar x) \coloneqq Q(x_0)^{-1} Q(x) \pmod{\pi}$ for all $x\in L$.
\item[(4)] The quadratic $F$-space $F\otimes_R L$ is anisotropic.
\end{enumerate}

\end{prop}	
\begin{proof} (1) $\Rightarrow$ (2).   By \cite[82:15]{O'M00}, one has $$L=Re_1\perp \cdots \perp Re_t \ \ \ \text{ with } \ \ \ Q(e_1)R=\cdots=Q(e_t)R=\mathfrak s(L) . $$
Suppose that there is a vector $x\in L\setminus \pi L$ such that $Q(x)\in \pi \mathfrak s(L)$. Write $x=\sum_{i=1}^t a_i e_i$ with $a_i\in R$. Since $x\notin \pi L$, we may assume without loss of generality that  $a_1\in R^\times$. Since $Q(x)Q(e_1)^{-1}\in \pi R$, by Hensel's lemma (see \cite[13:8]{O'M00}), there is $\xi\in R^\times$ such that
$Q(x) Q(e_1)^{-1}  \xi^2 +2 a_1 \xi +1 =0$. Then $y:=\xi x+e_1$ satisfies
\[
Q(y)=Q(\xi x+e_1)=Q(x)\xi^2+2\xi\langle x,e_1\rangle+Q(e_1)=Q(x) \xi^2 +2 a_1Q(e_1) \xi +Q(e_1) =0.
\]Moreover,
\[
\mathrm{ord}(Q(x) \xi +2 a_1Q(e_1))=\mathrm{ord}(Q(x)\xi^2+2a_1Q(e_1)\xi)=\mathrm{ord}(Q(e_1))<\mathrm{ord}(Q(x))=\mathrm{ord}(Q(x)\xi).
\]Hence
\[
\mathrm{ord}(2Q(x) \xi +2 a_1Q(e_1))=\mathrm{ord}(Q(x) \xi +2 a_1Q(e_1))<\mathrm{ord}(Q(x)\xi)=\mathrm{ord}(Q(x)).
\]This shows that
 \[
 \eta\coloneqq -2^{-1}\big(\xi Q(x)+a_1Q(e_1)\big)^{-1} Q(x) \in \pi R .
   \]Thus, the vector $z:=x+\eta y=(1+\eta\xi)x+\eta e_1$ satisfies
   \[
   Q(z)=Q(x)+2\eta\langle x,y\rangle+\eta^2Q(y)=Q(x)+2\eta(\xi Q(x)+a_1Q(e_1))=0
   \]and
   \[
   \langle y,z\rangle R=\langle y,x+\eta y\rangle R=(\xi Q(x)+a_1Q(e_1))R=Q(e_1)R=\mathfrak{s}(L).
   \]Hence, the sublattice $Ry+Rz$ is a hyperbolic binary lattice, and it splits $L$ by \cite[82:15]{O'M00}. A contradiction is derived.

(2) $\Rightarrow$ (3). Let $\bar x$ be any nonzero vector in $L/\pi L$, and choose a lifting $x\in L$ of $\bar x$. Then $x\in L\setminus \pi L$, so, by (2) we have
 $Q(x)R=\mathfrak s(L)=Q(x_0)R$. This implies $Q(x_0)^{-1} Q(x)\in R^\times$, and hence $\bar{Q}(\bar x)=Q(x_0)^{-1} Q(x)\mod \pi$ is nonzero in $R/\pi R$. One concludes that  $(L/\pi L,  \bar { Q } )$ is anisotropic over $R/\pi R$, as desired.

(3) $\Rightarrow$ (4).  Suppose that there is $y\in F\otimes_R L$ such that $Q(y)=0$. Then one can assume that $y\in L\setminus \pi L$. Then $\bar{y} \neq 0$ in $L/\pi L$ and $\bar{Q}(\bar{y})=0$. This is absurd when $(L/\pi L,  \bar { Q } )$ is anisotropic over $R/\pi R$.

(4) $\Rightarrow$ (1). This is obvious.
\end{proof} 	

We have the following analog of \cite[91:5]{O'M00}:

\begin{coro}\label{modular} If $L$ is a modular quadratic lattice over $R$ such that the quadratic space $F\otimes_R L$ over $F$ is anisotropic, then $\mathrm{O}(L)=\mathrm{O}(F\otimes_R L)$.
 \end{coro}

\begin{proof} For any reflection $\tau_x\in \mathrm{O}(F\otimes_R L)$, one can choose $x\in L\setminus \pi L$. Since $Q(x)R=\mathfrak s(L)$ by Proposition \ref{hyper}, one has $\tau_x\in \mathrm{O}(L)$. The result follows from \cite[43:3 Theorem]{O'M00}.
\end{proof}

The first assertion in the following result is contained in \cite[Chapter VI, \S 1, Proposition 1.9]{Lam05}.

\begin{coro} \label{2-modular}
Let $L_1$ and $L_2$ be modular quadratic lattices over $R$ with $\mathfrak s(L_1)=\pi^{r_1} R$ and $\mathfrak s(L_2)=\pi^{r_2} R$ such that $F\otimes_R L_1$ and $F\otimes_R L_2$ are anisotropic.

If $r_1\equiv r_2+1\pmod2$, then the space $F\otimes_R (L_1 \perp L_2)$ is anisotropic over $F$. Moreover, if $|r_1-r_2|=1$, then $$\mathrm{O} (L_1\perp L_2)=\mathrm{O} (F\otimes_R (L_1 \perp L_2)). $$
    \end{coro}
\begin{proof}Suppose that $r_1\equiv r_2+1\pmod 2$ and that $F\otimes_R (L_1 \perp L_2)$ is isotropic over $F$. Then $Q(x)=0$ for some nonzero vector $x\in F\otimes_R (L_1 \perp L_2)$. Let $v_1\in L_1\setminus \pi L_1$ and $v_2\in L_2 \setminus \pi L_2$ be such that $x= a_1 v_1 + a_2 v_2$ with $a_1, a_2 \in F$. Since $F\otimes_R L_1$ and $F\otimes_R L_2$ are anisotropic, we have $a_1a_2 \neq 0$ and $Q(x)= a_1^2 Q(v_1)+ a_2^2 Q(v_2)=0$. This implies $\mathrm{ord}(Q(v_1))\equiv\mathrm{ord}(Q(v_2))\pmod{2}$. But $\mathrm{ord}(Q(v_i))=r_i$ for $i=1,2$, by Proposition \ref{hyper}. A contradiction is thus derived.

Now suppose $|r_1-r_2|=1$. Without loss of generality, we assume $r_1< r_2$.

For any reflection $\tau_z\in \mathrm{O}(F\otimes_R (L_1\perp L_2))$, one can assume $z=b_1 z_1 + b_2 z_2$ with  $z_1\in L_1 \setminus \pi L_1$, $z_2\in L_2\setminus \pi L_2$ and $b_1, b_2\in R$ such that  at least one of $b_1$ and $b_2$ is in $R^\times$.

If $b_1\in R^\times$, then $\mathrm{ord}(Q(z))=r_1$ and
$$ \langle z, L_1\perp L_2 \rangle = \langle b_1 z_1, L_1\rangle + \langle b_2 z_2, L_2\rangle = b_1 \pi^{r_1} R + b_2 \pi^{r_2} R = \pi^{r_1} R\,, $$
by \cite[82:17]{O'M00}. This implies $\tau_{z} \in \mathrm{O}(L_1\perp L_2)$, as desired. Otherwise, we have $b_1\in \pi R$ and $b_2\in R^\times$. Since $r_2-r_1=1$, one has $ord(Q(z))=r_2$ and
$$ \langle z, L_1\perp L_2 \rangle = \langle b_1 z_1, L_1\rangle + \langle b_2 z_2, L_2\rangle = b_1 \pi^{r_1} R + b_2 \pi^{r_2} R = \pi^{r_2} R\,, $$
by \cite[82:17]{O'M00}. This again implies $\tau_{z} \in \mathrm{O}(L_1\perp L_2)$.
\end{proof}

\subsection{Sets of transporters and their spinor norms}\label{sec3.2}

In this subsection, let $M$ be a quadratic $R$-lattice and let $N\subseteq M$ be a (non-degenerate) sublattice. Put $V=F\otimes_RM$. Let $\theta:\mathrm{O}(V)\to F^\times/(F^{\times})^2$ denote the spinor norm map of the quadratic space $V$ (cf. \cite[\S\;55]{O'M00}). For any subset $X\subseteq\mathrm{O}(V)$, we often consider the image $\theta(X)$ as a subset of $F^{\times}$, which is a union of cosets of $(F^{\times})^2$ in $F^{\times}$.

Recall (from the introduction) that the set of transporters of $N$ in $M$ is the set
\[
X(M/N)=\{\sigma\in \mathrm{O}(V)\,|\,N\subseteq \sigma(M)\}.
\]The set of proper transporters is defined to be $X^+(M/N)=X(M/N)\cap \mathrm{O}^+(V)$.

We prove two lemmas which will allow reductions to simpler cases when studying  the sets $X(M/N)$, $X^+(M/N)$ and their spinor norms. As analogs of \cite[Lemmas 5.1 and 5.2]{HSX98}, these two lemmas treat the case of equal scale and the case of unequal scale respectively.

\begin{lemma}\label{lemma1}
Suppose that there are orthogonal decompositions of quadratic lattices $N=N_1\perp N'$ and $M=N_1\perp M'$ such that $N_1$ is modular with $\mathfrak{s}(N_1)=\mathfrak{s}(M)$.

Then $$X(M/N)=X(M'/N') \cdot \mathrm{O}(M) \ \ \ \text{and} \ \ \  X^+(M/N)=X^+(M'/N') \cdot \mathrm{O}^+(M)  . $$
\end{lemma}
\begin{proof}
Since $2$ is invertible in $R$, the modular lattice $N_1$ has an orthogonal basis over $R$, by \cite[82:15]{O'M00}. By induction on the rank of $N_1$, we may assume that $N_1=R x$ for some nonzero vector $x$. Note that $Q(x)R=\mathfrak{s}(N_1)=\mathfrak{s}(M)$.

For any $\sigma\in X(M/N)$, we consider the two vectors $u\coloneqq x-\sigma^{-1}x$  and $w\coloneqq x+\sigma^{-1}x$ in $M$. Since $Q(u)+Q(w)=4Q(x)$, we have  $\mathfrak{s}(M)=Q(x)R\subseteq  Q(u) R$  or  $\mathfrak{s}(M)=Q(x)R \subseteq Q(w)R$. This implies that $\tau_u\in \mathrm{O}(M)$ or $\tau_w\in \mathrm{O}(M)$. Moreover, one has
$$ 	\tau_u(x)=\sigma^{-1}x \ \ \ \text{and} \ \ \ \tau_w(x)=-\sigma^{-1}x$$ by a direct computation.

		If $\tau_u\in \mathrm{O}(M)$, then $ \sigma \tau_u(x)=x$ and
$$ N=N_1\perp N' \subseteq \sigma (M) = \sigma \tau_u (M) = \sigma \tau_u (N_1 \perp M' ) = \sigma\tau_u N_1 \perp \sigma\tau_u (M') = N_1 \perp \sigma\tau_u (M') . $$
Therefore $N' \subseteq  \sigma\tau_u (M')$. Namely,   $\sigma\tau_u \in X(M'/N')$ as desired. Furthermore, when $\sigma\in X^+(M/N)$, one can choose $x_1\in M'$  such that $Q(x_1)R =\mathfrak{s}(M')$. Then $\tau_u\tau_{x_1}\in \mathrm{O}^+(M)$ and $\sigma\tau_u\tau_{x_1}(x)=x$. Replacing $\tau_u$ with $\tau_u\tau_{x_1}$ in the above argument, one concludes that $\sigma\tau_u \tau_{x_1} \in X^+(M'/N')$.

If $\tau_w\in \mathrm{O}(M)$, then $\tau_w\tau_x\in \mathrm{O}^+(M)$ and $\sigma \tau_w\tau_x (x)=x$. Therefore
$$ N=N_1\perp N' \subseteq \sigma (M) = \sigma \tau_w \tau_x (M) = \sigma \tau_w \tau_x  (N_1 \perp M' ) = \sigma\tau_w\tau_x N_1 \perp \sigma\tau_w\tau_x (M') = N_1 \perp \sigma\tau_w\tau_x (M') $$ and so $N' \subseteq  \sigma\tau_w\tau_x (M')$. This implies
\[
\sigma= (\sigma\tau_w\tau_x) (\tau_x\tau_w)\in X(M'/N') \cdot \mathrm{O}^+(M).
 \]Moreover, when $\sigma\in X^+(M/N)$, one has $\sigma\tau_w\tau_x \in X^{+}(M'/N')$. This completes the proof.	\end{proof}

\begin{lemma}\label{reduction2} Suppose that $\mathfrak{s}(N)\neq \mathfrak{s}(M)$, and let  $M= M_1\perp\cdots\perp M_t$ be a Jordan splitting of $M$.
	
	If $M_1$ is split by a hyperbolic binary lattice, then $\theta(X(M/N))=\theta(X^+(M/N))=F^\times$.
	
	Otherwise $N\subseteq M^*\coloneqq\pi M_1\perp\cdots\perp M_t $ and $$X(M/N)=X(M^*/N) \ \ \ \text{ and } \ \ \ X^+(M/N)=X^+(M^*/N) . $$
\end{lemma}
\begin{proof}
First assume that $M_1$ is split by a hyperbolic binary lattice. Since $X^+(M/N)\subseteq X(M/N)$, we only need to prove $\theta(X^+(M/N))=F^\times$. In this case one can write
	$$M_1=M_0\perp (Re+Rf)   \ \ \ \text{with} \ \ \
Q(e)=Q(f)=0 \ \text{ and } \ \langle e, f\rangle R = \mathfrak s (M_1) . $$  For any $\alpha\in R^\times$, one has
$$ \tau_{e+f}\cdot \tau_{e+\alpha f} \in \mathrm{O}^+(Re+Rf) \subseteq \mathrm{O}^+(M_1)\subseteq \mathrm{O}^+(M) \ \ \ \text{and} \ \ \ \theta(\tau_{e+f} \cdot \tau_{e+\alpha f})=\alpha  (F^{\times})^2 . $$
This shows that $ \theta(\mathrm{O}^+(M)) \supseteq R^\times (F^{\times})^2$. Since $\mathrm{O}^+(M)\subseteq X^+(M/N)\mathrm{O}^+(M)=X^+(M/N)$, it is sufficient to show that
\[
\theta(X^+(M/N))\cap \pi R^\times (F^{\times})^2\neq\emptyset.
\]
Set
$$ M':= M_0 \perp  (Re+R\pi f) \perp M_2\perp \cdots \perp M_t  \subseteq M . $$
Since $\mathfrak{s}(N)\neq \mathfrak{s}(M)$, one has $\mathfrak s(Re+R\pi f)=\pi\mathfrak s(M_1) = \pi \mathfrak s(M) \supseteq \mathfrak s(N)$. By Theorem \ref{localrep}, there exists an isometry $\sigma_0 \in \mathrm{O}(V)$ such that $\sigma_0(N) \subseteq M'$.

If $\sigma_0\in\mathrm{O}^+(V)$, then $\sigma_0^{-1}\in X^+(M'/N)\subseteq X^+(M/N)$. On the other hand, since $\tau_{e+\pi f}\sigma_0(N)\subseteq \tau_{e+\pi f}(M')\subseteq M'\subseteq M$, we have  $\sigma_0^{-1} \tau_{e+\pi f}\in X(M/N)$ and hence $\sigma_0^{-1} \tau_{e+\pi f} \tau_{e+f} \in X^+(M/N)$ as $\tau_{e+f}\in \rO(M)$.
Therefore both $\theta (\sigma_0^{-1})$ and $\theta (\sigma_0^{-1} \tau_{e+\pi f} \tau_{e+f})$ lie in $\theta(X^+(M/N))$.
We have
\[
\theta (\sigma_0^{-1} \tau_{e+\pi f} \tau_{e+f})=Q(e+f) \cdot Q(e+\pi f) \cdot \theta(\sigma_0^{-1})= \pi \theta(\sigma_0^{-1})(F^{\times})^2.
\]So it follows that $\theta(X^+(M/N))\cap \pi R^\times (F^{\times})^2\neq\emptyset$,  as desired.

If $\sigma_0\notin\mathrm{O}^+(V)$, then $\sigma_0':=\tau_{e+\pi f}\sigma_0\in\mathrm{O}^+(V)$ and hence $(\sigma_0')^{-1}=\sigma_0^{-1}\tau_{e+\pi f}$ lies in $X^+(M'/N)$. Replacing $\sigma_0$ with $\sigma_0'$ in the previous arguments we see again that $\theta(X^+(M/N))\cap \pi R^\times (F^{\times})^2\neq\emptyset$. This proves the first assertion of the lemma.

Now consider the case that $M_1$ is not split by a hyperbolic binary lattice. Note that the equality $X(M/N)=X(M^*/N)$ implies $X^+(M/N)=X^+(M^*/N)$, so we only need to prove the first equality.

For any $\sigma\in X(M/N)$ and $y\in N$, we have $\sigma^{-1}(y)\in \sigma^{-1}(N)\subseteq M$. So we can write
\[
\sigma^{-1}(y)=\sum_{i=1}^t x_i \quad\text{with }   \ x_i\in M_i \text{ for each }  1\leq i\leq t.
\] Then
$$ \sum_{i=1}^t Q(x_i) = Q(\sigma^{-1}(y))=Q(y) \in \mathfrak s(N) . $$
If $x_1\in M_1\setminus \pi M_1$, then $Q(x_1) R= \mathfrak s(M_1)= \mathfrak s(M)$ by Proposition \ref{hyper}, and for $i\ge 2$, we have
$Q(x_i)R\subseteq \mathfrak{s}(M_i)\subseteq \pi\mathfrak{s}(M_1)$. Thus,
$$\mathfrak s(M)= Q(x_1)R = \bigg(\sum_{i=1}^t Q(x_i)\bigg) R = Q(y) R \subseteq \mathfrak s(N) . $$ This contradicts the assumption  $\mathfrak{s}(N)\neq \mathfrak{s}(M)$. Therefore $x_1\in \pi M_1$ and $\sigma\in X(M^*/N)$, proving that $N\subseteq M^*$ and $X(M/N)=X(M^*/N)$ as desired.
\end{proof}

\begin{prop}\label{proposition4} The image sets $\theta(X(M/N))$ and $\theta(X^+(M/N))$ are both subgroups of $F^\times$.
\end{prop}
\begin{proof}  By applying Lemma \ref{reduction2} repeatedly, one concludes that either $$\theta(X(M/N))=\theta(X^+(M/N))=F^\times$$ or there is a quadratic sublattice $M^\dag \subseteq M$  containing $N$ with $\mathfrak s(M^\dag)=\mathfrak s(N)$ such that $$X(M/N) = X(M^\dag/N) \ \ \ \text{and}  \ \ \ X^+(M/N) = X^+(M^\dag/N).$$ The result follows from Lemma \ref{lemma1}, by induction on $\rank (M)$.
\end{proof}

The following result is a refinement of Lemma \ref{reduction2} and turns out to be more suitable for induction in some cases.

\begin{prop} \label{red-ref} Suppose that $\mathfrak s(N)=\pi^{t_1}R\neq \mathfrak{s}(M)$. Let
 $$M= M_1\perp\cdots\perp M_m \ \ \ \text{where} \ \ \ \mathfrak s(M_i)= \pi^{s_i} R \ \ \ \text{for} \ \ 1\leq i\leq m \ \ \ \text{and} \ \ \ s_1< \cdots < s_m $$
 be a given Jordan splitting of $M$.

If $\mathop{\mathlarger{\mathlarger{\perp}}}\limits_{s_j< t_1} (F \otimes_R M_j)$ is isotropic over $F$, then $$\theta(X(M/N))=\theta(X^{+}(M/N))=F^\times . $$

Otherwise, one has $N\subseteq M^\ddag$, $X(M/N)=X(M^{\ddag}/N)$ and $X^{+}(M/N)=X^{+}(M^{\ddag}/N)$, where
$$M^{\ddag}\coloneqq \left(\mathop{\mathlarger{\mathlarger{\perp}}}\limits_{\substack{s_j\leq t_1\\s_j\equiv t_1-1\pmod2}} \pi^{(t_1-s_j-1)/2}M_j \right) \mathlarger{\mathlarger{\perp}} \left(\mathop{\mathlarger{\mathlarger{\perp}}}\limits_{\substack{s_j\leq t_1\\s_j\equiv t_1\pmod2}} \pi^{(t_1-s_j)/2} M_j\right)\mathlarger{\mathlarger{\perp}}  \left(\mathop{\mathlarger{\perp}}\limits_{s_j>t_1} M_j \right). $$

\end{prop}

\begin{proof} We prove this result by induction on the number $m$ of Jordan components of $M$.

When $m=1$, the result follows from Proposition \ref{hyper} and Lemma \ref{reduction2}.

For $m>1$, we can assume that $F\otimes_R M_1$ is anisotropic over $F$ by Proposition \ref{hyper} and Lemma \ref{reduction2}. Having six possibilities in total, we shall finish the proof case by case.

\noindent {\bf Case 1.} $s_1\equiv s_2\pmod2$ and $t_1<s_2$.

In this case we have $s_1\equiv t_1\pmod 2$. Because otherwise one can apply Lemma \ref{reduction2} repeatedly to obtain a sublattice $M'$ of $M$ containing $N$ with $\mathfrak{s}(M')$ strictly contained in $\mathfrak{s}(N)$, which is absurd. Thus,
\[
M^{\ddag}=\pi^{(t_1-s_1)/2} M_1\perp M_2\perp \cdots \perp M_m,
\] and we obtain the desired result  by applying Lemma \ref{reduction2} repeatedly.

\noindent {\bf Case 2.} $s_1\equiv s_2\pmod2$ and $t_1\geq s_2$.

In this case we can apply  Lemma \ref{reduction2} repeatedly to get
\[
 X(M/N)=X(M'/N)  \quad \text{and}  \quad X^+(M/N)=X^+(M'/N)
 \] with
\[
 M':=\big(\pi^{(s_2-s_1)/2} M_1\perp M_2\big)\perp \cdots \perp M_m.
\]
Since the number of Jordan components of $M'$ is $m-1$, the result now follows by induction.

\noindent {\bf Case 3.} $s_1\equiv s_2+1 \equiv t_1 \pmod2$ and $t_1<s_2$.

In this case $M^{\ddag}=\pi^{(t_1-s_1)/2} M_1\perp M_2\perp\cdots \perp M_m$ and the result is obtained  by applying Lemma \ref{reduction2} repeatedly.

\noindent {\bf Case 4.} $s_1\equiv s_2+1 \equiv t_1 \pmod2$ and $s_2<t_1\leq s_3$.

Putting $M':=\pi^{(s_2-s_1+1)/2} M_1\perp M_2\perp \cdots \perp M_m$, we have
$$ X(M/N)=X(M'/N) \quad \text{and}\quad X^+(M/N)=X^+(M'/N)$$ by applying Lemma \ref{reduction2} repeatedly. The first Jordan component of $M'$ is  $M_2$.

If $$\mathop{\mathlarger{\mathlarger{\perp}}}\limits_{s_j< t_1} (F \otimes_R M_j)= (F\otimes_R M_1)\perp (F\otimes_R M_2) $$ is isotropic over $F$, one concludes that $F\otimes_R M_2$ is isotropic over $F$ by Corollary \ref{2-modular}. Therefore
$$\theta(X(M'/N))=\theta(X^{+}(M'/N))=F^\times$$ by Proposition \ref{hyper} and Lemma \ref{reduction2}. This yields the desired result.

Otherwise, one can apply  Lemma \ref{reduction2} repeatedly to conclude, noticing that
\[
M^{\ddag}=\pi^{(t_1-s_1)/2} M_1\perp \pi^{(t_1-s_2-1)/2} M_2 \perp M_3\perp \cdots \perp M_m
\]
in the current case.

\noindent {\bf Case 5.} $s_1 \equiv s_2+1 \equiv t_1 \pmod2$ and $s_3 < t_1$.

Similar arguments as in the previous case show that $$ X(M/N)=X(M'/N) \quad\text{and}\quad X^+(M/N)=X^+(M'/N)$$
with $M'=\pi^{(s_2-s_1+1)/2} M_1 \perp M_2 \perp \cdots \perp M_m$, and that if $F\otimes_R M_2$ is isotropic over $F$, then
$\theta(X(M'/N))=\theta(X^{+}(M'/N))=F^\times$.

We may now assume that $F\otimes_R M_2$ is anisotropic over $F$. Then, again by applying Lemma \ref{reduction2} repeatedly, we obtain
\[
X(M/N)=X(M'/N)=X(M''/N)\quad \text{and}\quad X^+(M/N)=X^+(M'/N)=X^+(M''/N)
\]
where
\[
M''=
 \begin{cases} \pi^{(s_3-s_1)/2} M_1 \perp \pi^{(s_3-s_2-1)/2} M_2 \perp M_3 \perp \cdots \perp M_m  \ & \text{if $s_1\equiv s_3 \pmod2$}, \\
   \pi^{(s_3-s_1-1)/2} M_1 \perp \pi^{(s_3-s_2)/2} M_2 \perp M_3 \perp \cdots \perp M_m   \ & \text{if $s_1\equiv s_3+1 \pmod2$}.
\end{cases}
\] Since the number of Jordan components of $M''$ is less than $m$, the result follows by induction.

\noindent {\bf Case 6.} $s_1+1\equiv s_2\equiv t_1 \pmod2$.

In this case we have $t_1\geq s_2$. Otherwise the same reasoning as in Case 1 would lead to a contradiction.

If $t_1= s_2$, then $M^{\ddag}=\pi^{(t_1-s_1-1)/2} M_1 \perp M_2 \perp \cdots \perp M_m$ and we can conclude by using Lemma \ref{reduction2} repeatedly.

Next let us assume that $s_2<t_1\leq s_3$. We may first use Lemma \ref{reduction2} to get
$$ X(M/N)=X(M'/N) \quad \text{and}\quad X^+(M/N)=X^+(M'/N)$$ with $M':=\pi^{(s_2-s_1+1)/2} M_1 \perp M_2 \perp \cdots \perp M_m$.
The  first Jordan component of $M'$ is $M_2$.

If $$\mathop{\mathlarger{\mathlarger{\perp}}}\limits_{s_j< t_1} (F \otimes_R M_j)= (F\otimes_R M_1) \perp (F\otimes_R M_2)$$ is isotropic over $F$, then so is   $F\otimes_R M_2$ over $F$, by Corollary \ref{2-modular}. In this case, one conclude that $$\theta(X(M/N))=\theta(X^{+}(M/N))=F^\times$$ by Proposition \ref{hyper} and Lemma \ref{reduction2}.

Otherwise, repeated applications of Lemma \ref{reduction2} yield
\[
X(M/N)=X(M'/N)=X(M^{\ddag}/N)\quad\text{and}\quad X^+(M/N)=X^+(M'/N)=X^+(M^{\ddag}/N),
\]noticing that
\[
M^{\ddag}=\pi^{(t_1-s_1-1)/2} M_1 \perp \pi^{(t_1-s_2)/2} M_2 \perp \cdots \perp M_m
\] in the present case.

Finally, assume that $s_3< t_1$. Putting $M':=\pi^{(s_2-s_1+1)/2} M_1 \perp M_2 \perp \cdots \perp M_m$, we can deduce from  Lemma \ref{reduction2}  that
\[
X(M/N)=X(M'/N)\quad\text{and}\quad X^+(M/N)=X^+(M'/N),
\]and from Proposition \ref{hyper}  that $\theta(X(M/N))=\theta(X^{+}(M/N))=F^\times$ if  $F\otimes_R M_2$ is isotropic over $F$.

Now consider the case that $F\otimes_R M_2$ is anisotropic over $F$. Then, by using Lemma \ref{reduction2} repeatedly we obtain
\[
X(M/N)=X(M''/N)\quad\text{and}\quad X^+(M/N)=X^+(M''/N),
\]where
\[
M'':=
\begin{cases} \pi^{(s_3-s_1)/2} M_1 \perp \pi^{(s_3-s_2-1)/2} M_2 \perp M_3 \perp \cdots \perp M_m \ \ \ & \text{if $s_1\equiv s_3 \pmod2$}, \\
 \pi^{(s_3-s_1-1)/2} M_1 \perp \pi^{(s_3-s_2)/2} M_2 \perp M_3 \perp \cdots \perp M_m   \ \ \ & \text{if $s_2\equiv s_3 \pmod2$}.
\end{cases}
\]Since the numbers of Jordan components of $M''$ is less than $m$, the result follows by induction.
\end{proof}

\subsection{Norm principles for integral spinor norms}\label{sec3.3}

 In \cite{S69} and  \cite{K71}, Scharlau and Knebusch  established two kinds of norm principle for quadratic spaces over fields.  Gille \cite{G97} generalized  these two norm principles in terms of isogenies between reductive groups and gave an alternative proof of Knebusch's norm principle. Taking the viewpoint of homomorphisms from an algebraic group to a commutative algebraic group, Merkurjev in \cite{M95} and Barquero and Merkurjev in \cite{BM02} formulated a norm principle which includes both Scharlau's and Knebusch's norm principles as special examples, and proved the norm principle for all reductive groups except possibly when  the Dynkin diagram contains $D_n$ with $n\geq 4$, $E_6$ or $E_7$. Bhaskhar, Chernousov and Merkurjev \cite{BCM19} further studied the norm principle for groups of type $D_n$ and showed that the norm principle over complete discrete valuation fields follows from the norm principle over their residue fields.  On the other hand, Ojanguren, Panin and Zainoulline proved Knebusch's norm principle for unimodular quadratic forms over a semi-local ring where 2 is invertible in \cite{OPZ04}.

In this subsection, we prove results that can be viewed as norm principles for the abstract groups $\theta(X(M/N))$ and $\theta(X^+(M/N))$ (cf. Prop. \ref{proposition4}), as a crucial step in the proof of our main theorem. Our results (Theorems \ref{normlattice} and \ref{relative}) indicate that certain integral versions of the norm principle can be expected in nice situations.

We begin with the norm principle for integral orthogonal groups. We need a formula for computing local integral spinor norms groups, which is essentially due to Kneser \cite[Satz 3]{Kn56}.

Recall that $R$ denotes a complete discrete valuation ring in which $2$ is invertible with fraction field $F$. For a quadratic $R$-lattice $L$, by a \emph{norm generator} of $L$ we mean a vector $z\in L$ such that $Q(z)R=\mathfrak{s}(L)$.

\begin{prop} \label{Kne56}
 Let $M=M_1\perp\cdots\perp M_t$ be a Jordan splitting of a quadratic $R$-lattice $M$. For each $1\le j\le t$, Define
   \[
   \mathfrak{M}_i:=\{Q(x_i)\,|\, x_i \text{ is a norm generator of } M_i\}\quad\text{ for each } 1\le i\le t
   \]and $\mathfrak{M}:=\bigcup_{1\le i\le t}\mathfrak{M}_i$.   Let $\mathfrak{M}^+$ be the subset of $F^\times$ consisting of elements that can be written as a product of an even number of elements in $\mathfrak{M}$.

 \begin{enumerate}
   \item[(1)] As subgroups of $F^\times$, $\theta(\mathrm{O}(M))$ is generated by $\mathfrak{M}\cdot (F^\times)^2$ and $\theta(\rO^+(M))=\mathfrak{M}^+\cdot (F^\times)^2$.
   \item[(2)] Choose a norm generator $x_i\in M_i$ for each Jordan component $M_i$.
  Then
  \[
  \theta(\mathrm{O}(M))= \prod_{i=1}^t \theta (\mathrm{O}(M_i))
  \quad\text{and}\quad
  \theta(\mathrm{O}^+(M))=\bigcup_{\substack{k\text{ even}\\ 1\leq i_1<\cdots<i_k\leq t}} Q(x_{i_1})\cdots Q(x_{i_k})\left(\prod_{i=1}^t\theta(\mathrm{O}^+(M_i))\right) .
  \]
 \end{enumerate}
	\end{prop}
\begin{proof} (1) One can check that the proofs of \cite[Hilfssatz 2 and Satz 3]{Kn56} are valid in our case (where the residue field of $R$ need not be finite), so the result follows.

(2) This is clear from (1).
\end{proof}

By Proposition \ref{Kne56}, the computation of integral spinor norms reduces to each Jordan component.
For modular quadratic lattices the following proposition provides some more precise information.

\begin{prop} \label{modular-spinor}
 Let $M$ be a modular quadratic lattice over $R$ with $\mathfrak s(M)=\pi^r R$ and put $V=F\otimes_R M$.

Then
  \[
  \theta (\mathrm{O} (M))\subseteq R^\times (F^\times)^2 \cup \pi^r R^\times (F^\times)^2 \quad \text{and} \quad \theta(\mathrm{O}^+(M))\subseteq R^\times (F^\times)^2 .
  \] If moreover $V$ is isotropic, then equality holds in the above inclusion relations.
 \end{prop}
\begin{proof}
The first assertion is immediate from Proposition \ref{Kne56}.

Let us assume that $V$ is isotropic. By  Proposition \ref{hyper}, $M$ is split by a hyperbolic binary lattice, so that there exist vectors $x,y\in M$ such that  $Q(x)=Q(y)=0$
 and $\langle x, y\rangle = \pi^r$. For any $\varepsilon\in R^\times$, one can verify that the reflections $\tau_{x+\varepsilon y}$ and $\tau_{x+y}$ are in $\mathrm{O}(M)$. Since
 $Q(x+\varepsilon y) \cdot Q(x+y)=  \varepsilon (2\pi^r)^2$ and $Q(x+y)=2 \pi^r$,  one obtains
$$\theta (\mathrm{O} (M))\supseteq R^\times (F^\times)^2 \cup \pi^r R^\times (F^\times)^2 \ \ \ \text{and} \ \ \ \theta(\mathrm{O}^+(M)) \supseteq R^\times (F^\times)^2 . $$
This completes the proof.
\end{proof}

Now we are ready to prove the norm principle for integral orthogonal groups.

\begin{thm} \label{normlattice}
 Let $R$ be a complete discrete valuation ring in which $2$ is invertible with fraction field $F$. Let $E/F$ be a finite extension and let $S$ be the integral closure of $R$ in $E$. Let $M$ be a quadratic lattice over $R$ and put $\widetilde{M} = S\otimes_RM$.

 We have $ N_{E/F}(\theta(\mathrm{O}(\widetilde{M})))\subseteq\theta(\mathrm{O}(M))$ and $N_{E/F}(\theta(\mathrm{O}^+(\widetilde{M})))\subseteq\theta(\mathrm{O}^+(M))$.
	\end{thm}
	\begin{proof}
		Let $M= M_1\perp\cdots\perp M_t$ be a Jordan splitting of $M$ and choose a norm generator $x_i\in M_i$ for each component $M_i$. Then
 $$\widetilde{M}=\widetilde{M_1}\perp\cdots\perp\widetilde{M_t} \ \ \  \text{where} \ \ \ \widetilde{M_i}=S\otimes_RM_i  $$ is a Jordan splitting of $\widetilde{M}$ with $x_i\in \widetilde{M_i}$ and $Q(x_i)S = \mathfrak s(\widetilde{M_i})$ for $1\leq i \leq t$. Since
$$N_{E/F}(Q(x_{i_1})\cdots Q(x_{i_k}))\in
			\begin{cases}
				 (F^{\times})^2 \ \ \  & \text{when $[E:F]$ is even},\\
				 Q(x_{i_1})\cdots Q(x_{i_k})(F^{\times})^{2} \ \ \  & \text{when $[E:F]$ is odd,}
			\end{cases} $$
it suffices by Proposition \ref{Kne56} to prove $$N_{E/F}(\theta(\mathrm{O}^+(\widetilde{M})))\subseteq\theta(\mathrm{O}^+(M)) \ \ \ \text{and} \ \ \ N_{E/F}(\theta(\mathrm{O}(\widetilde{M})))\subseteq\theta(\mathrm{O}(M))$$ when $M$ is modular.

If $F\otimes_R M$ is isotropic, then $E\otimes_S \widetilde{M}$ is also isotropic. The result follows from Proposition \ref{modular-spinor}.

If $F\otimes_R M$ is anisotropic, then $$\theta (\mathrm{O} (M))=\theta (\mathrm{O} (F\otimes_R M)) \ \ \ \text{and} \ \ \ \theta(\mathrm{O}^+(M))=\theta(\mathrm{O}^+(F\otimes_R M))$$
by Corollary \ref{modular}. Therefore
$$ N_{E/F}\big(\theta(\mathrm{O}(\widetilde{M}))\big)\subseteq N_{E/F}\big(\theta(\mathrm{O}(E\otimes_S \widetilde{M}))\big) \subseteq \theta (\mathrm{O} (F\otimes_R M))= \theta (\mathrm{O} (M)) $$
and
$$ N_{E/F}\big(\theta(\mathrm{O}^+(\widetilde{M}))\big)\subseteq N_{E/F}\big(\theta(\mathrm{O}^+(E\otimes_S \widetilde{M}))\big) \subseteq \theta (\mathrm{O}^+ (F\otimes_R M))= \theta (\mathrm{O}^+ (M)) $$ by
Knebusch's norm principle \cite[Satz]{K71} (see also \cite[Chapter VII, \S5, Theorem 5.1]{Lam05}).
 \end{proof}

The following norm principle is a key ingredient in our proof of Theorem \ref{intro}.

\begin{thm} \label{relative}
Let $R$ be a complete discrete valuation ring in which $2$ is invertible with fraction field $F$. Let $E/F$ be a finite extension and let $S$ be the integral closure of $R$ in $E$. Let $M$ be a quadratic lattice over $R$, $N\subseteq M$ a (non-degenerate) sublattice, and put $\widetilde{M} = S\otimes_RM$, $\widetilde{N}=S\otimes_RN$.

Then we have
 $$N_{E/F}(\theta(X(\widetilde{M}/\widetilde{N})))\subseteq\theta(X(M/N)) \ \ \ \text{and} \ \ \ N_{E/F}(\theta(X^+(\widetilde{M}/\widetilde{N})))\subseteq\theta(X^+(M/N)) .$$
\end{thm}	
\begin{proof} We prove the first inclusion and the second one follows from the same arguments.

If $\mathfrak{s}(N)=\mathfrak{s}(M)$, one can choose an $x\in N$ with $Q(x)R=\mathfrak{s}(M)$ so that there are splittings
$$M=R x\perp M' \ \ \ \text{and}  \ \ \ N=R x\perp N' . $$ Then $\widetilde{M}= S x\perp\widetilde{M'}$ and $\widetilde{N}= S x\perp\widetilde{N'}$. By Lemma \ref{lemma1}, one has $$\theta(X(M/N))=\theta(X(M'/N')) \cdot \theta(\rO(M)) \ \ \ \text{and} \ \ \ \theta(X(\widetilde{M}/\widetilde{N}))=\theta(X(\widetilde{M'}/\widetilde{N'}))\cdot \theta(\rO(\widetilde{M})) .$$  The result follows from Theorem \ref{normlattice} and induction on the rank of  $M$.

From now on we assume that $\mathfrak{s}(N)\neq \mathfrak{s}(M)$. We write $\mathfrak{s}(N)=\pi^{t_1} R$ and choose a Jordan splitting
\[
M=M_1\perp M_2\perp\cdots\perp M_m \ \ \ \text{with} \ \  \mathfrak s(M_i)=\pi^{s_i}R \ \ \text{for} \ \ 1\leq i\leq m .
\]Note that $$\widetilde{M}=\widetilde{M_1}\perp\cdots\perp\widetilde{M_m}$$ is a Jordan splitting of $\widetilde{M}$.

Let $k_F$ and $k_E$ be the residue fields of $F$ and $E$ respectively. There is an unramified subextension $K/F$ in $E/F$ such that
 the residue field $k_K$ of $K$ is the separable closure of $k_F$ in $k_E$, by \cite[Chapter III, \S 5, Cor. 3 of Thm. 3]{Serre-GTM67}. So we only need to prove the theorem in two special cases: the case with $E/F$ unramified and the case where the residue field extension $k_E/k_F$ is purely inseparable (or trivial).

%\bigskip

\medskip

\noindent{\bf Case 1.} $E/F$ is unramified.

 If $M_1$ is split by a hyperbolic binary lattice, then $\theta(X(M/N))=F^\times$ by Lemma \ref{reduction2}, and the norm principle holds automatically.

 So we assume that $M_1$ is not split by a hyperbolic binary lattice. Then $N\subseteq M^*=\pi M_1\perp\cdots\perp M_m$ and $X(M/N)=X(M^*/N)$
by Lemma \ref{reduction2}.

If $\widetilde{M_1}$ is not split by a hyperbolic binary lattice, then	by
Lemma \ref{reduction2} we have
$$\widetilde{N}\subseteq\pi\widetilde{M_1}\perp\cdots\perp\widetilde{M_t}=\widetilde{M^*} \ \ \ \text{and} \ \ \ X(\widetilde{M}/\widetilde{N})=X(\widetilde{M^*}/\widetilde{N}).$$
 (Note that the uniformizer $\pi$ of $F$ is also a uniformizer of $E$ in the current case.) The result follows from induction on $\mathfrak s(M)$ (or equivalently, induction on the integer $s_1$ determined by $\mathfrak{s}(M)=\pi^{s_1}R$).

If $\widetilde{M_1}$ is split by a hyperbolic binary lattice, then $[E:F]$ is even by Proposition \ref{hyper} and the classical Springer theorem for quadratic spaces \cite{Spr52}. Thus, using Lemma \ref{reduction2} we obtain
$$N_{E/F}(\theta(X(\widetilde{M}/\widetilde{N})))=N_{E/F}(E^\times) = N_{E/F} (\pi^{\mathbb{Z}} \cdot S^\times)\subseteq N_{E/F}(S^\times)\cdot (F^\times)^2.$$  Note that
 $\theta (\mathrm{O}^+(\widetilde{M_1})) = S^\times (E^\times)^2$ by Proposition \ref{modular-spinor}. Hence, by Theorem \ref{normlattice},
$$N_{E/F}(S^\times) \subseteq N_{E/F}\big(\theta (\mathrm{O}^+(\widetilde{M_1}))\big)\subseteq \theta (\mathrm{O}^+(M_1)) \subseteq \theta(X(M/N)),  $$ which implies the desired result.

\medskip

\noindent{\bf  Case 2.} The residue field extension $k_E/k_F$ of  $E/F$ is purely inseparable.

Since $2$ is invertible in $R$, the residue degree $[k_E: k_F]$ is odd. Thus, the ramification index $e$ of $E/F$ is even if $[E:F]=e[k_E:k_F]$ is even.

By Proposition \ref{red-ref}, one can assume that $\mathop{\mathlarger{\mathlarger{\perp}}}\limits_{s_j< t_1}(F \otimes_R M_j)$ is anisotropic over $F$ and that    $X(M/N)=X(M^{\ddag}/N)$, where
$$M^{\ddag}= \left(\mathop{\mathlarger{\mathlarger{\perp}}}\limits_{\substack{s_j\leq t_1\\s_j\equiv t_1-1\pmod2}} \pi^{(t_1-s_j-1)/2}M_j \right) \mathlarger{\mathlarger{\perp}} \left(\mathop{\mathlarger{\mathlarger{\perp}}}\limits_{\substack{s_j\leq t_1\\s_j\equiv t_1\pmod2}} \pi^{(t_1-s_j)/2} M_j\right)\mathlarger{\mathlarger{\perp}}  \left(\mathop{\mathlarger{\perp}}\limits_{s_j>t_1} M_j \right). $$

We fix a uniformizer $\pi_E$ of $E$.

\medskip

\noindent {\bf Subcase 2.1.} $s_1=t_1-1$.

In this case the same reasoning as in Case 1 of Proposition \ref{red-ref} shows that $s_2=t_1$. Let
$$N=N_1 \perp N_2\perp \cdots \perp N_n \ \ \ \text{with} \ \ \mathfrak s(N_j)=\pi^{t_j}R \ \ \text{for} \ \ 1\leq j\leq n $$ be a Jordan splitting of $N$. We have  $N\subseteq M^*= \pi M_1\perp M_2 \perp \cdots \perp M_m$ by Lemma \ref{reduction2}. Note that there is a Jordan splitting of $M^*$ in which $M_2$ is the Jordan component of largest scale. So by Theorem \ref{localrep}, $N_1$ can be embedded into $M_2$. By  Corollary \ref{uniqueness} and \cite[82:15]{O'M00}, we may assume that  $M_2$ has an orthogonal splitting $M_2=N_1\perp B_2$.

Putting
\[
M':=B_2\perp \pi M_1 \perp M_3\perp \cdots \perp M_m\quad\text{and}\quad N':=N_2 \perp \cdots \perp N_n,
\]we can deduce from Lemmas \ref{lemma1} and \ref{reduction2} that
\[
X(M/N)=X(M^*/N)=X(M'/N') \cdot \mathrm{O} (M^*).
\] We can assume that $F\otimes_R B_2$ is anisotropic over $F$, for otherwise $\theta(X(M'/N'))=F^\times$ by Proposition \ref{hyper} and Lemma \ref{reduction2}.

We distinguish three cases according to the isotropy properties of the $E$-spaces $E\otimes_S \widetilde{M_1}$ and $E\otimes_S \widetilde{B_2}$.

\noindent {\bf Subcase 2.1 (a).} $E\otimes_S \widetilde{M_1}$ is isotropic over $E$.

We have $\theta (X(\widetilde{M}/\widetilde{N}))= E^\times$ by Lemma \ref{reduction2}, and
$\theta (\mathrm {O}(E\otimes_S \widetilde{M_1}))=E^\times$ by \cite[55:2a]{O'M00}. Therefore,
$$N_{E/F}(\theta (X(\widetilde{M}/\widetilde{N}))) = N_{E/F} (\theta (\mathrm {O}(E\otimes_S \widetilde{M_1}))) \subseteq \theta (\mathrm {O} (F\otimes_R M_1))= \theta (\mathrm {O} (M_1)) \subseteq \theta (X(M/N)) $$
by Knebusch's norm principle \cite[Satz]{K71} and Corollay \ref{modular}.

\noindent {\bf Subcase 2.1 (b).} $E\otimes_S \widetilde{M_1}$ is anisotropic over $E$ but $E\otimes_S \widetilde{B_2}$ is isotropic over $E$.

In this case $[E: F]$ is even by the classical Springer theorem \cite{Spr52}. It follows that the ramification index $e$ of $E/F$ is even. Thus, we can use Lemmas \ref{lemma1} and \ref{reduction2} to obtain
\[
\begin{split}
  &X(\widetilde{M}/\widetilde{N})= X(\pi_E^{e/2} \widetilde{M_1} \perp (\widetilde{N_1}\perp \widetilde{B_2}) \perp \cdots \perp \widetilde{M_m}/\widetilde{N_1}\perp \cdots \perp \widetilde{N_n}) \\
  =\;& X(\pi_E^{e/2} \widetilde{M_1} \perp \widetilde{B_2} \perp \widetilde{M_3} \perp \cdots \perp \widetilde{M_m}/\widetilde{N_2}\perp \cdots \perp \widetilde{N_n}) \cdot \mathrm{O} (\pi_E^{e/2} \widetilde{M_1} \perp \widetilde{M_2}\perp \cdots \perp \widetilde{M_m}).
\end{split}
\]Now
$$\theta (X(\pi_E^{e/2} \widetilde{M_1} \perp \widetilde{B_2} \perp \widetilde{M_3} \perp \cdots \perp \widetilde{M_m}/\widetilde{N_2}\perp \cdots \perp \widetilde{N_n}))= E^\times $$ by Lemma \ref{reduction2}, and $\theta (\mathrm {O}(E\otimes_S \widetilde{B_2}))=E^\times$ by \cite[55:2a]{O'M00}. So one concludes that
$$N_{E/F}(\theta (X(\widetilde{M}/\widetilde{N} )))=N_{E/F}(\theta (\mathrm {O}(E\otimes_S \widetilde{B_2}))) \subseteq \theta (\mathrm {O} (F\otimes_R B_2))= \theta (\mathrm {O} (B_2)) \subseteq \theta (X(M/N)) $$
by Knebusch's norm principle   and Corollay \ref{modular}.

\noindent {\bf Subcase 2.1 (c).} Both $E\otimes_S \widetilde{M_1}$ and $E\otimes_S \widetilde{B_2}$ are anisotropic over $E$ (allowing the case $B_2=0$).

We have
\begin{equation}\label{NP1}
  \begin{split}
   &X(\widetilde{M}/\widetilde{N} )= X(\pi_E^{[(e+1)/2]} \widetilde{M_1} \perp (\widetilde{N_1}\perp \widetilde{B_2})\perp \cdots \perp \widetilde{M_m}/\widetilde{N_1}\perp \cdots \perp \widetilde{N_n})\\
   =\;& X(\pi_E^{[(e+1)/2]} \widetilde{M_1} \perp \widetilde{B_2}  \perp \cdots \perp \widetilde{M_m}/\widetilde{N_2}\perp \cdots \perp \widetilde{N_n})\cdot\mathrm{O} (\pi_E^{[(e+1)/2]} \widetilde{M_1} \perp \widetilde{M_2}\perp \cdots \perp \widetilde{M_m})
  \end{split}
\end{equation}
 by  Lemmas \ref{lemma1} and \ref{reduction2}, where $[(e+1)/2]$ is the largest integer $\le (e+1)/2$.

Suppose that  $E\otimes_S (\pi_E^{[(e+1)/2]} \widetilde{M_1} \perp \widetilde{B_2}) $ is isotropic over $E$. Since $F\otimes_R(M_1\perp B_2)$ is anisotropic (by Corollary \ref{2-modular}), $[E:F]$ must be even by the classical Springer theorem, and it follows that $e$ is even. Then $\pi_E^{[(e+1)/2]} \widetilde{M_1} $ and
$\widetilde{B_2}$ have the same scale,
$$ \theta (X(\pi_E^{[(e+1)/2]} \widetilde{M_1} \perp \widetilde{B_2} \perp \widetilde{M_3} \perp \cdots \perp \widetilde{M_m}/\widetilde{N_2}\perp \cdots \perp \widetilde{N_n}))=E^\times $$ by Lemma \ref{reduction2}, and $\theta (\mathrm{O} (E\otimes_S (\pi_E^{[(e+1)/2]} \widetilde{M_1} \perp \widetilde{B_2})))=E^\times$ by \cite[55:2a]{O'M00}.
Therefore
\[
\begin{split}
 & N_{E/F}\big(\theta (X(\widetilde{M}/\widetilde{N}))\big)= N_{E/F} \big(\theta (\mathrm{O} (E\otimes_S (\pi_E^{[(e+1)/2]} \widetilde{M_1} \perp \widetilde{B_2})))\big) \\
\subseteq\;& \theta (\mathrm{O} (F\otimes_R (M_1\perp B_2)))= \theta (\mathrm{O} (M_1\perp B_2)) \subseteq \theta(X(M/N))
\end{split}
\]by Knebusch's norm principle and Corollary \ref{2-modular}.

Now assume that $E\otimes_S (\pi_E^{[(e+1)/2]} \widetilde{M_1} \perp \widetilde{B_2}) $ is anisotropic over $E$.
Since $t_2\geq t_1+1=s_1+2$, one obtains that
 \begin{equation}\label{NP2}
 \begin{split}
  & X(\pi_E^{[(e+1)/2]} \widetilde{M_1} \perp \widetilde{B_2} \perp \widetilde{M_3} \perp \cdots \perp \widetilde{M_m}/\widetilde{N_2}\perp \cdots \perp \widetilde{N_n}) \\
=\;& X(\pi_E^e \widetilde{M_1} \perp \pi_E^{[e/2]}\widetilde{B_2} \perp \widetilde{M_3} \perp \cdots \perp \widetilde{M_m}/\widetilde{N_2}\perp \cdots \perp \widetilde{N_n})
 \end{split}
 \end{equation} by applying  Lemma \ref{reduction2} repeatedly. Additionally, Lemma \ref{reduction2} also implies  that
\begin{equation}\label{NP3}
 \begin{split}
 &X(\widetilde{M'}/\widetilde{N'})=X(\widetilde{\pi M_1} \perp \widetilde{B_2} \perp \widetilde{M_3} \perp \cdots \perp \widetilde{M_m}/\widetilde{N_2}\perp \cdots \perp \widetilde{N_n})\\
=\;&X(\pi_E^e \widetilde{M_1} \perp \widetilde{B_2} \perp \widetilde{M_3} \perp \cdots \perp \widetilde{M_m}/\widetilde{N_2}\perp \cdots \perp \widetilde{N_n})\\
=\;& X(\pi_E^e \widetilde{M_1} \perp \pi_E^{[e/2]}\widetilde{B_2} \perp \widetilde{M_3} \perp \cdots \perp \widetilde{M_m}/\widetilde{N_2}\perp \cdots \perp \widetilde{N_n}).\end{split}
 \end{equation}
 It is now sufficient to prove the relation
\begin{equation}\label{NP4}
N_{E/F}\big(\theta(\mathrm{O} (\pi_E^{[(e+1)/2]} \widetilde{M_1} \perp \widetilde{M_2}\perp \cdots \perp \widetilde{M_m}))\big) \subseteq
N_{E/F}\big(\theta(\mathrm{O} (\pi_E^{e} \widetilde{M_1} \perp \widetilde{M_2}\perp \cdots \perp \widetilde{M_m}))).
\end{equation}
Indeed, combining \eqref{NP1}, \eqref{NP2}, \eqref{NP3} and \eqref{NP4}, we get
\[
\begin{split}
&N_{E/F}\big(\theta(X(\widetilde{M}/\widetilde{N}))\big)=N_{E/F}\big(\theta(X(\widetilde{M'}/\widetilde{N'}))\cdot\theta(\mathrm{O} (\pi_E^{[(e+1)/2]} \widetilde{M_1} \perp \widetilde{M_2}\perp \cdots \perp \widetilde{M_m}))\big)\\
\subseteq\;&N_{E/F}\big(\theta(X(\widetilde{M'} /\widetilde{N'}))\cdot\theta(\mathrm{O} (\widetilde{\pi M_1} \perp \widetilde{M_2}\perp \cdots \perp \widetilde{M_m}))\big)=N_{E/F}\big(\theta(X(\widetilde{M'} /\widetilde{N'}))\cdot\theta(\mathrm{O} (\widetilde{M^*}))\big).
\end{split}
\]Thus, the desired norm principle follows from induction on $\rank(M)$ and Theorem \ref{normlattice}, in view of the fact $X(M/N)=X(M'/N')\cdot\mathrm{O}(M^*)$.

If $e=1$, then \eqref{NP4} holds trivially. If $e$ is odd and $e>1$, then $\pi_E^{[(e+1)/2]} \widetilde{M_1} \perp \widetilde{M_2}\perp \cdots \perp \widetilde{M_m}$ has exactly $m$ components in its Jordan splitting. Hence \eqref{NP4} follows from Proposition \ref{Kne56}.

Now assume $e$ is even. Then $\pi_E^{[(e+1)/2]} \widetilde{M_1}$ and $\widetilde{M_2}$ have the same scale, which is strictly contained in $\mathfrak{s}(\widetilde{M_3})$.
Thus, by Proposition \ref{Kne56}, \eqref{NP4} follows from the following relation:
\begin{equation}\label{NP5}
  N_{E/F}\big(\theta(\mathrm{O} (\pi_E^{[(e+1)/2]} \widetilde{M_1} \perp \widetilde{M_2}))\big) \subseteq
(F^\times)^2.
\end{equation}
For the $S$-lattice $L:=\pi_E^{[(e+1)/2]} \widetilde{M_1} \perp \widetilde{M_2}$, its scale has the form $\mathfrak{s}(L)=\pi_E^{r}S$ for some even natural number $r$, since $e$ is even. Applying Proposition \ref{modular-spinor} to $L$ we get
\[
N_{E/F}(\rO(L))\subseteq N_{E/F}(S^\times\cdot (E^\times)^2)\subseteq N_{E/F}(S^{\times})\cdot (F^\times)^2.
\](The acute reader will notice that this argument actually shows that the left-hand side of (\ref{NP4}) is contained in $(F^\times)^2$.)

To finish the proof of \eqref{NP5}, it suffices to show that $N_{E/F}(S^\times)\subseteq (F^\times)^2$. For any $x\in R$ (resp. $x\in S$), let $\bar x\in k_F$ (resp. $\bar x\in k_E$) denote its canonical image in the residue field of $R$ (resp. $S$). For any $\alpha\in S^{\times}$, there exists an element $\beta\in R^\times$ such that
$\bar{\alpha}^{[k_E:k_F]}=\bar{\beta}$, because $k_E/k_F$ is purely inseparable. Since $[k_E:k_F]$ is odd, $\bar\alpha\bar\beta^{-1}$ is a square in $k_E$. Hence, by Hensel's lemma, $\alpha=\beta\gamma^2$ for some $\gamma\in S^\times$. Thus
\[
N_{E/F}(\alpha)=N_{E/F}(\beta)N_{E/F}(\gamma)^2=\beta^{[E:F]}N_{E/F}(\gamma)^2\in (F^\times)^2
\]as desired.

This finishes the proof in the case $s_1=t_1-1$.

\medskip

\noindent {\bf Subcase 2.2.} $s_1\leq t_1-2$.

In this case the definition of $M^{\ddag}$ implies that  $\mathfrak s(M^{\ddag})$ is strictly contained in $\mathfrak s(M)$.

If $ \mathop{\mathlarger{\mathlarger{\perp}}}\limits_{s_j< t_1} (E \otimes_S \widetilde{M_j})$ is isotropic over $E$, then
$$ \theta (X(\widetilde{M}/\widetilde{N}))=\theta \big(\mathrm{O} (\mathop{\mathlarger{\mathlarger{\perp}}}\limits_{s_j< t_1} (E \otimes_S \widetilde{M_j}))\big) = E^\times$$  by Proposition \ref{red-ref} and \cite[55:2a]{O'M00}.  Knebusch's norm principle yields
\[
N_{E/F}(\theta (\mathrm{O} (\mathop{\mathlarger{\mathlarger{\perp}}}\limits_{s_j< t_1}(E \otimes_S \widetilde{M_j}))) ) \subseteq \theta (\mathrm{O} (\mathop{\mathlarger{\mathlarger{\perp}}}\limits_{s_j< t_1} (F \otimes_R M_j) )),
\] and
\begin{align*}  &\theta (\mathrm{O}(\mathop{\mathlarger{\mathlarger{\perp}}}\limits_{s_j< t_1} (F\otimes_R M_j)) ) =\theta \left( \mathrm{O}\left((\mathop{\mathlarger{\mathlarger{\perp}}}\limits_{\substack{s_j< t_1\\s_j\equiv t_1-1\pmod2}} \pi^{(t_1-s_j-1)/2}M_j) \mathlarger{\mathlarger{\perp}} (\mathop{\mathlarger{\mathlarger{\perp}}}\limits_{\substack{s_j< t_1\\s_j\equiv t_1\pmod2}} \pi^{(t_1-s_j)/2}M_j)\right)\right) \end{align*}
by Corollary \ref{2-modular}. Since
$$\mathrm{O}\left((\mathop{\mathlarger{\mathlarger{\perp}}}\limits_{\substack{s_j< t_1\\s_j\equiv t_1-1\pmod2}} \pi^{(t_1-s_j-1)/2}M_j) \mathlarger{\mathlarger{\perp}} (\mathop{\mathlarger{\mathlarger{\perp}}}\limits_{\substack{s_j< t_1\\s_j\equiv t_1\pmod2}} \pi^{(t_1-s_j)/2}M_j)\right) \subseteq \mathrm{O} (M^{\ddag}),$$
one concludes that
$$N_{E/F} (\theta (X(\widetilde{M}/\widetilde{N} )))\subseteq \theta ( \mathrm{O}(M^{\ddag}))\subseteq \theta(X(M^{\ddag}/N))= \theta(X(M/N)).$$

If $\mathop{\mathlarger{\mathlarger{\perp}}}\limits_{s_j< t_1} (E \otimes_S \widetilde{M_j})$ is anisotropic over $E$, then  by Proposition \ref{red-ref},
$$X(\widetilde{M}/\widetilde{N})=X\big(P_1\perp P_2 \perp (\mathop{\perp}\limits_{s_j >t_1} \widetilde{M_j})/\widetilde{N}\big)$$ where
 $$ P_1 = \mathop{\mathlarger{\mathlarger{\perp}}}\limits_{\substack{s_j\leq t_1\\es_j\equiv et_1-1\pmod2}} \pi_E^{(et_1-es_j-1)/2}\widetilde{M_j} \ \ \ \text{and} \ \ \ P_2= \mathop{\mathlarger{\mathlarger{\perp}}}\limits_{\substack{s_j\leq t_1\\es_j\equiv et_1\pmod2}} \pi_E^{(et_1-es_j)/2} \widetilde{M_j}.$$
 Since
 $$\widetilde{N}\subseteq  P_1\perp P_2 \perp (\mathop{\perp}\limits_{s_j >t_1} \widetilde{M_j}) \subseteq  \widetilde{M^{\ddag}}\subseteq \widetilde{M}  $$ and
 $$ X\big(P_1\perp P_2 \perp (\mathop{\perp}\limits_{s_j >t_1} \widetilde{M_j})/\widetilde{N})\subseteq X(\widetilde{M^{\ddag}}/\widetilde{N} \big) \subseteq X(\widetilde{M}/\widetilde{N} ) $$ by definition, one concludes $X(\widetilde{M^{\ddag}}/\widetilde{N} ) = X(\widetilde{M}/\widetilde{N} )$.
 The result follows from induction on $\mathfrak s(M)$.
\end{proof}
	
\section{Global theory for embeddings of quadratic lattices}\label{sec4}

We now explain the main ingredients in the global theory of quadratic lattices that will be used in the proof of our main theorem.

As in the classical case of a ring of integers in a number field, using the adelic language and the genus theory turns out to be an appropriate way to study quadratic lattices over a general Dedekind domain. This approach has been used recently in \cite{HuLiuTian23} to study the integral Hasse principle for representations of quadratic lattices.

In this section, $R$ denotes a general Dedekind domain whose fraction field $F$ has characteristic $\mathrm{char}(F)\neq 2$.

\subsection{Adelic language and genus theory}\label{sec4.1}

For the reader's convenience, let us introduce some notations and recall some definitions that will be frequently used in this section.

We denote by $\Spm(R)$ the set of maximal ideals of $R$.  For each $\mathfrak{p}\in\Spm(R)$, in this paper we let $R_{\mathfrak{p}}$ (rather than $\widehat{R_{\fp}}$) denote the completion of $R$ at $\mathfrak{p}$ for simplicity, and let $F_{\mathfrak{p}}$ be the fraction field of $R_{\mathfrak{p}}$.

Let $M$ be a (non-degenerate) quadratic lattice over $R$ and put $V=F\otimes_RM$. Let $\theta: \mathrm{O}(V)\to F^\times/(F^\times)^2$ denote the spinor norm map for $V$.

 For each $\mathfrak{p}\in\Spm(R)$, write $V_\mathfrak p=F_\mathfrak p\otimes_F V$ and $M_\mathfrak p= R_\mathfrak p\otimes_R M$. As already mentioned in the introduction, we have the orthogonal group $\mathrm{O}(V_\mathfrak p)$ and the special orthogonal group $\mathrm{O}^+(V_\mathfrak p)$ of $V_\mathfrak p$, and their integral versions
$\mathrm{O}(M_\mathfrak p)$ and $\mathrm{O}^+(M_\mathfrak p)$ for the $R_{\mathfrak{p}}$-lattice $M_{\mathfrak{p}}$. The spinor norm map for the $F_{\fp}$-space $V_{\fp}$ is denoted by $\theta_{\fp}: \rO(V_{\fp})\to F^\times_{\fp}/(F^\times_{\fp})^2$.

Let $\bfO(V)$ (resp. $\bfO^+(V)$) denote the \emph{orthogonal group} (resp. \emph{special orthogonal group}) of the quadratic space $V$ as an algebraic group (\cite[\S23, p.~348]{KMRT98}). The groups $\rO(V)$ and $\rO^+(V)$ can be viewed as the groups of $F$-points of $\bfO(V)$ and $\bfO^+(V)$ respectively.

Let $\rO_{\bfA}(V)$ denote the group of adelic points of $\bfO(V)$ over $R$, i.e., the restricted product of
$\{\rO(V_{\fp})\}_{\fp\in\Spm(R)}$ with respect to $\{\rO(M_{\fp})\}_{\fp\in\Spm(R)}$, given explicitly by
\[
\rO_{\bfA}(V):=\bigg\{(\sigma_{\fp})\in\prod_{\fp\in\Spm(R)}\rO(V_{\fp})\,\Big|\, \sigma_{\fp}(M_{\fp})=M_{\fp} \text{ for almost all } \fp\bigg\}\,.
\]This group depends only on the quadratic space $V$ and is independent of the choice of the lattice $M$. Via the diagonal embedding we may view $\rO(V)$ as a subgroup of $\rO_{\bfA}(V)$ (cf. \cite[101:4]{O'M00}).

The group of adelic points $\rO^+_{\bfA}(V)$ of $\bfO^+(V)$ is defined similarly.

We also define
\[
\rO_{\bfA}(M):=\prod_{\mathfrak p\in \Spm(R)} \mathrm{O}(M_\mathfrak p)\quad\text{and}\quad \rO^+_{\bfA}(M):=\prod_{\mathfrak p\in \Spm(R)} \mathrm{O}^+(M_\mathfrak p)\,.
\]These are clearly subgroups of $\rO_{\bfA}(V)$.

Let $\bfI_F$ denote the \emph{id\`ele group} of $F$ with respect to the places in $\Spm(R)$, i.e., $\bfI_F$ is the restricted product of $\{F_{\fp}^\times\}_{\fp\in\Spm(R)}$ with respect to $\{R_{\fp}^\times\}_{\fp\in\Spm(R)}$. Since  $F$ has characteristic $\neq 2$, for almost all $\fp\in\Spm(R)$ we have $2\in R_{\fp}^{\times}$ and $M_{\fp}$ is a unimodular $R_{\fp}$-lattice. For any such $\fp$ we have $\theta_{\fp}(\rO(M_{\fp}))\subseteq R_{\fp}^\times (F_{\fp}^\times)^2$, by Proposition \ref{modular-spinor}. Therefore, the local spinor norm maps induce an \emph{adelic spinor norm map} $\theta_{\bfA}:\rO_{\bfA}(V)\to \bfI_F/\bfI_F^2$. For any subset $X\subseteq\rO_{\bfA}(V)$, we regard $\theta_{\bfA}(X)$ as a subset of $\bfI_F$, which is a union of cosets of $\bfI_F^2$ in $\bfI_F$.

We put
\[
\rO'_{\bfA}(V):=\rO^+_{\bfA}(V)\cap\ker\big(\theta_{\bfA}: \rO_{\bfA}(V)\longrightarrow\bfI_F/\bfI_F^2\big).
\]
Note that $\rO'_{\bfA}(V)$ contains the commutator subgroup of $\rO_{\bfA}(V)$. Hence, any subgroup $S\le \rO_{\bfA}(V)$ containing $\rO'_{\bfA}(V)$ is normal in $\rO_{\bfA}(V)$ and the quotient group $\rO_{\bfA}(V)/S$ is an abelian 2-torsion group.

\medskip

For every $\underline{\sigma}=(\sigma_{\fp})\in\rO_{\bfA}(V)$, the same argument as in \cite[\S\,101.D, p.~297]{O'M00} shows that there is a unique lattice $K$ on $V$ such that $K_{\fp}=\sigma_{\fp}M_{\fp}$ for all $\fp\in\Spm(R)$. We denote this lattice $K$ by $\underline{\sigma}M$. From the definition it is clear that
$\underline{\sigma}M=M$ iff $\underline{\sigma}\in\rO_{\bfA}(M)$.

\begin{defn}\label{defn4.1} With notation as above, the \emph{genus} $\gen(M)$ of $M$ is defined as the orbit of $M$ under the action of $\rO_{\bfA}(V)$, i.e.,
\[
\gen(M):= \{\underline{\sigma}M\,|\,\underline{\sigma}\in\rO_{\bfA}(V)\}\,.
\]The orbits of $M$ under the actions of the subgroups $\rO(V)\rO'_{\bfA}(V)$, $\rO^+(V)\rO'_{\bfA}(V)$, $\rO(V)$ and $\rO^+(V)$, denoted by
$\spn(M)$, $\spn^+(M)$, $\cls(M)$ and $\cls^+(M)$, are called the \emph{spinor genus}, \emph{proper spinor genus}, \emph{class} and \emph{proper class} of $M$ respectively.
\end{defn}

\begin{remark}\label{orb+}  For any $\fp\in\Spm(R)$, one has  $\tau_{u}\in\rO(M_{\fp})$ for any norm generator $u$ of $M_{\fp}$. From this one concludes that $\mathrm{gen}(M)$ is also the orbit of $M$ under the action of  $\mathrm{O}_{\mathbf A}^{+}(V)$.
\end{remark}

It is clear that $\spn^+(M)\subseteq\spn(M)$, $\cls^+(M)\subseteq\cls(M)$,
\[
\cls(M)\subseteq \spn(M)\subseteq \gen(M)\quad \text{and}\quad \mathrm{cls}^+(M)\subseteq \mathrm{spn}^+(M) \subseteq \mathrm{gen}(M).
\]In the case where  $F$ is a global field and $\Spm(R)$ contains almost all places of $F$, it is well known that the number of classes,  the number of proper classes, the number of spinor genera, and the number of proper spinor genera in $\gen(M)$ are all finite (cf. \cite[103:4]{O'M00}). However, these numbers may not be finite in general.

Letting the double coset $\mathrm{O}(V) \underline{\sigma} \rO_{\bfA}(M)$ correspond to the class $\mathrm{cls}(\underline{\sigma}(M))$ establishes a natural bijection
\[
\mathrm{O}(V) \big\backslash \mathrm{O}_{\mathbf A}(V)\big/ \rO_{\bfA}(M) \overset{\sim}{\longrightarrow}\{\cls(L)\,|\,L\in\gen(M)\}
\] between the double quotient of $\bfO_{\bfA}(V)$ modulo the subgroups $\rO(V), \prod_{\mathfrak p\in \Spm(R)} \mathrm{O}(M_\mathfrak p)$
and the set of lattice classes in $\gen(M)$. Similarly, there is a natural bijection
\[
\begin{split}
\mathrm{O}^+(V) \big\backslash \mathrm{O}^+_{\mathbf A}(V)\big/\rO_{\bfA}^+(M) &\overset{\sim}{\longrightarrow}\{\cls^+(L)\,|\,L\in\gen(M)\}\,,\\
\mathrm{O}^+(V) \underline{\sigma}\rO^+_{\bfA}(M)&\longmapsto \cls^+(\underline{\sigma}M).
\end{split}
\]
We may also consider spinor genera and proper spinor genera contained in $\gen(M)$. Analogously, we have the natural bijections
\begin{equation}\label{spin-genus}
\begin{split}
\big(\mathrm{O}(V)\rO'_{\bfA}(V)\big) \big\backslash \mathrm{O}_{\mathbf A}(V)\big/ \rO_{\bfA}(M) &\overset{\sim}{\longrightarrow}\{\spn(L)\,|\,L\in\gen(M)\}\,,\\
\big(\mathrm{O}^+(V)\rO'_{\bfA}(V)\big) \big\backslash \mathrm{O}^+_{\mathbf A}(V)\big/\rO^+_{\bfA}(M) &\overset{\sim}{\longrightarrow}\{\spn^+(L)\,|\,L\in\gen(M)\}.
\end{split}
\end{equation}
Since $\mathrm{O}(V)\rO'_{\bfA}(V)\rO_{\bfA}(M)$ and $\mathrm{O}^+(V)\rO'_{\bfA}(V)\rO^+_{\bfA}(M)$
are normal subgroups in $\rO_{\bfA}(V)$, the double quotients appearing in the two bijections in \eqref{spin-genus} can be rewritten as the quotient groups
\[
\rO_{\bfA}(V)\big/\big(\rO(V)\rO'_{\bfA}(V)\rO_{\bfA}(M)\big)\quad\text{and}\quad \rO_{\bfA}^+(V)\big/\big(\mathrm{O}^+(V)\rO'_{\bfA}(V)\rO^+_{\bfA}(M)\big),
\]
which are abelian 2-torsion groups. Thus, via the bijections in \eqref{spin-genus}, the spinor genera contained in a fixed genus $\gen(M)$ form an abelian 2-torsion group, and similarly for the proper spinor genera in $\gen(M)$.

\subsection{Spinor genera and classes containing a given sublattice}\label{sec4.2}

Now let $N$ be a quadratic lattice in $V=F\otimes_{R}M$. As analogs of the sets $X(M/N)$ and $X^+(M/N)$, we can define the set of \emph{adelic transporters} $X_{\bfA}(M/N)$ and the set  of \emph{proper adelic transporters} $X^+_{\bfA}(M/N)$ of $N$ in $M$  by
\[
X_{\mathbf A}(M/N)= \{ \underline{\sigma} \in \mathrm{O}_{\mathbf A}(V): N \subseteq \underline{\sigma}M \} \quad\text{and} \quad X_{\mathbf A}^+(M/N)= \{ \underline{\sigma} \in \mathrm{O}_{\mathbf A}^+(V): N \subseteq \underline{\sigma}M \}.
\]When $N\subseteq M$ is a sublattice, we have clearly
\begin{equation}\label{4.2.1}
\begin{split}
  \rO_{\bfA}(M)&\subseteq X_{\bfA}(M/N)=X_{\bfA}(M/N)\rO_{\bfA}(M)\,,\\
  \text{and }\ \rO^+_{\bfA}(M)&\subseteq X^+_{\bfA}(M/N)=X^+_{\bfA}(M/N)\rO^+_{\bfA}(M).
\end{split}
\end{equation}
For any lattice $L$ on $V$, let us write $N\hookrightarrow \spn(L)$ (resp. $N\hookrightarrow\spn^+(L)$) and say that $N$ \emph{embeds into} $\spn(L)$ (resp. $\spn^+(L)$) if $N$ is contained in some lattice in $\spn(L)$ (resp. in $\spn^+(L)$).
%
%Similarly, we write $N\hookrightarrow \cls(L)$ (resp. $N\hookrightarrow\cls^+(L)$) and say that $N$ embeds into $\cls(L)$ (resp. $\cls^+(L)$) if $N$ is contained in some lattice in $\cls(L)$ (resp. in $\cls^+(L)$). Note that $N\hookrightarrow \cls(L)$ is equivalent to saying that there exists an embedding $N\rightarrow L$ of $R$-lattices.

\begin{lemma}[{\cite[Lemma\;2.3]{HSX98}}]\label{spinor-N}
  With notation as above, let $\underline{\tau}\in \rO_{\bfA}(V)$.

  Then $N\hookrightarrow\spn(\underline{\tau}M)$ if and only if $\underline{\tau}\in \rO(V)\rO'_{\bfA}(V)X_{\bfA}(M/N)$.
  Similarly, if $\underline{\tau}\in \rO^+_{\bfA}(V)$, then $N\hookrightarrow\spn^+(\underline{\tau}M)$ if and only if $\underline{\tau}\in \rO^+(V)\rO'_{\bfA}(V)X^+_{\bfA}(M/N)$.
\end{lemma}
\begin{proof}
  This is immediate from the relevant definitions.
\end{proof}

By Lemma \ref{spinor-N}, we have natural bijections
\[
  \begin{split}
    \{\spn(L)\,|\, L\in\gen(M),\; N \hookrightarrow \spn(L)\}&\cong (\rO(V)\rO'_{\bfA}(V))\big\backslash\rO(V)\rO'_{\bfA}(V)X_{\bfA}(M/N)\big/\rO_{\bfA}(M),\\
    \{\spn^+(L)\,|\, L\in\gen(M),\; N \hookrightarrow \spn^+(L)\}&\cong (\rO^+(V)\rO'_{\bfA}(V))\big\backslash\rO^+(V)\rO'_{\bfA}(V)X^+_{\bfA}(M/N)\big/\rO^+_{\bfA}(M).
  \end{split}
\]
When $2\in R^\times$, the double quotient in the second  bijection has a natural group structure, by Theorem \ref{relspinor} below, which is an analog of \cite[Theorem 2.1]{HSX98}.

\begin{thm} \label{relspinor}
 Suppose that $2\in R^\times$. Then $\mathrm{O}'_{\bfA}(V) X^+_{\mathbf A} (M/N)$ is a (normal) subgroup of $\mathrm{O}^+_{\mathbf A}(V)$.
\end{thm}
\begin{proof}
It is sufficient to show that $\theta_{\mathbf{A}}(X^+_{\mathbf A}(M/N))$ is a subgroup of $\mathbf{I}_F$.

 For all $\underline{\rho_1}=(\rho_{1,\fp}),\underline{\rho_2}=(\rho_{2,\fp})\in X^+_\bfA(M/N)$, we have $a_\fp\coloneqq\theta_\fp(\rho_{1,\fp})\theta_\fp(\rho_{2,\fp})\in \theta_\fp(X^+(M_\fp/N_\fp))$ for all $\fp$, by Proposition \ref{proposition4}. Moreover, $a_\fp\in R_\fp^\times$ and $M_\fp$ is unimodular for almost all $\fp$. For any such $\fp$, if $F_\fp \otimes_{R_\fp} M_\fp$ is isotropic, one has $a_\fp\in R_\fp^\times=\theta_{\fp}(\mathrm{O}^+(M_{\fp})) \subseteq \theta_{\fp}(X^+(M_{\fp}/N_{\fp}))$, where the equality is ensured by Proposition \ref{modular-spinor}; while if $F_{\fp} \otimes_{R_\fp} M_\fp$ is anisotropic, then as $\mathrm{O}^+(M_\fp)=\mathrm{O}^+(V_\fp)=X^+(M_\fp/N_\fp)$ by Corollary \ref{modular}, one has $a_\fp\in \theta_{\fp}(X^+(M_{\fp}/N_{\fp}))=\theta_{\fp}(\mathrm{O}^+(M_{\fp}))$. Therefore we can find $\underline{\rho}=(\rho_\fp)\in X^+_\bfA(M/N)$ with $\theta_\fp(\rho_\fp)=a_\fp$ for all $\fp$, and thus $\theta_\bfA(\underline{\rho})=\theta_\bfA(\underline{\rho_1})\theta_\bfA(\underline{\rho_2})$, showing that $\theta_{\mathbf{A}}(X^+_{\mathbf{A}}(M/N))$ is closed under multiplication. Similarly one can see that it is also closed under inversion.
\end{proof}

\begin{remark}
Similar arguments as in the proof of Theorem \ref{relspinor} show that $\theta_{\mathbf{A}}(X_{\mathbf{A}}(M/N))$  is a subgroup of $\bfI_F$. But this may not be sufficient to guarantee that $\mathrm{O}'_{\bfA}(V)X_{\mathbf A}(M/N)$ is a subgroup of $\rO_{\bfA}(V)$, because an element in $\rO_{\bfA}(V)$ with trivial spinor norm need not lie in the subgroup $\rO'_{\bfA}(V)$ (which consists only of elements with trivial spinor norm in $\rO^+_{\bfA}(V)$).

From this point of view, the set of proper adelic transporters behaves better than the set of all adelic transporters.

See also Remark \ref{remark5.2} for a similar phenomenon.
\end{remark}

We will see that  spinor genera (resp. proper spinor genera) of lattices on $V$  can be replaced by their classes (resp. proper classes) if
the strong approximation property holds for the spin group $\mathbf{Spin}(V)$ of the quadratic space $V=F\otimes_RM$. (See e.g. \cite[\S\;23, p.349]{KMRT98} for the definition of the spin group.)

\begin{defn}\label{SAdef}
Let $G$ be a connected linear algebraic group over $F$. We can choose a finite subset $T_0\subseteq\Spm(R)$ and a group scheme $\mathcal{G}$ defined over $B:=\mathrm{Spec}(R)\setminus T_0$ such that $\mathcal{G}\times_{B}F\cong G$. We can define the group of adelic points $G_{\bfA}$ of $G$ to be the restricted product of
$\{G(F_{\fp})\}_{\fp\in\Spm(R)}$ with respect to $\{\mathcal{G}(R_{\fp})\}_{\fp\in\Spm(R)\setminus T_0}$. This definition is independent of the choice of the integral model $\mathcal{G}/B$.

We say that $G$ satisfies \emph{strong approximation} over $R$ if the diagonal image of $G(F)$ is dense in $G_{\bfA}$ in the adelic topology, or more explicitly, if  for every nonempty finite subset $T\subseteq\Spm(R)$ containing $T_0$ and every nonempty $\fp$-adic open subset $U_{\fp}\subseteq G(F_{\fp})$ for each $\fp\in T$, one has
  \[
  G(F)\cap \bigg(\prod_{\fp\in T}U_{\fp}\times\prod_{\fp\in\Spm(R)\setminus T}\mathcal{G}(R_{\fp})\bigg)\neq\emptyset
  \]when $G(F)$ is viewed as a subset of $\prod_{\fp\in\Spm(R)}G(F_{\fp})$ via the diagonal embedding.
\end{defn}

The following result has been essentially proved in \cite[Theorem\;3.2]{HuLiuTian23}.
\begin{lemma}\label{SAcoro}
  Suppose that the spin group $G:=\mathbf{Spin}(V)$ satisfies strong approximation over $R$.

  Then for every lattice $L$ with $F\otimes_RL=V$ we have
  \[
  \rO(V)\rO_{\bfA}(L)=\rO(V)\rO'_{\bfA}(V)\rO_{\bfA}(L)\quad\text{and}\quad \rO^+(V)\rO^+_{\bfA}(L)=\rO^+(V)\rO'_{\bfA}(V)\rO^+_{\bfA}(L).
  \]Consequently, $\spn(L)=\cls(L)$ and $\spn^+(L)=\cls^+(L)$.
\end{lemma}
\begin{proof}
  Let $\phi: G\to\bfO^+(V)$ be the natural homomorphism which identifies $G=\mathbf{Spin}(V)$ with the simply connected cover of the special orthogonal group
  $\bfO^+(V)$. By taking Galois cohomology it induces the following commutative diagram with exact rows
\[
  \xymatrix{
G(F)  \ar[r]^{\phi} \ar[d] & \mathrm{O}^+(V)    \ar[r]^{\theta}  \ar[d]  &  F^*/F^{*2} \ar[d] \\
 G_{\mathbf A} \ar[r]^{\phi_{\mathbf A}} & \mathrm{O}_{\mathbf A}^+(V)  \ar[r]^{\theta_{\mathbf A}} & \mathbf I_F/ \mathbf I_F^2
}
\]and this shows in particular that $\rO'_{\bfA}(V)=\phi_{\bfA}(G_{\bfA})$. The inverse image $\phi_{\bfA}^{-1}(\rO_{\bfA}^+(L))$ is an open subgroup of $G_{\bfA}$, and its cosets in $G_{\bfA}$ have nonempty intersections with $G(F)$ by the strong approximation property of $G$. Hence
$G_{\bfA}=G(F)\cdot\phi^{-1}_{\bfA}(\rO_{\bfA}^+(L))$, and it follows that
\setcounter{equation}{+1}
\begin{equation}\label{4.2.2}
\rO'_{\bfA}(V)=\phi_{\bfA}(G_{\bfA})\subseteq\phi(G(F))\rO_{\bfA}^+(L)\subseteq \rO^+(V)\rO^+_{\bfA}(L).
\end{equation}
This implies the first assertion in the lemma.

Since $\spn(L)$ is the orbit of $L$ under $\rO(V)\rO'_{\bfA}(V)\rO_{\bfA}(L)$ and $\cls(L)$ is the orbit of $L$ under $\rO(V)\rO_{\bfA}(L)$, we can deduce from the first assertion that $\spn(L)=\cls(L)$. Similarly, $\spn^+(L)=\cls^+(L)$.
\end{proof}

\begin{remark}\label{grp}
  Assume that $\mathbf{Spin}(V)$ satisfies strong approximation over $R$. Combing Lemma \ref{SAcoro} with the discussions in the end of \S\;\ref{sec4.1}, we see that the classes of lattices contained in $\gen(M)$ form an abelian 2-torsion group, and similarly for the proper classes in $\gen(M)$. Moreover, we can further conclude that the classes (or proper classes) in which the given sublattice $N$ embeds form a subgroup, at least when $2\in R^\times$.
\end{remark}
	
\begin{remark}\label{rmkSA}
  Strong approximation over a general Dedekind domain $R$ for semisimple simply connected groups over its fraction field was established by Harder in \cite{Harder67} for quasi-split groups. This result is generalized by Gille in \cite[Corollaire 5.11]{G09}. (See also \cite[Théorème 2.2]{CT}.)

  The case we need  is the case of the spin group $\mathbf{Spin}(V)$ of the quadratic $F$-space $V$. If $V$ is isotropic  of dimension $\ge 3$, then this spin group is rational over $F$ by \cite{P79}. In this case, we can apply \cite[Corollaire 5.11]{G09} to deduce that $\mathbf{Spin}(V)$ satisfies strong approximation over $R$.
\end{remark}

\section{Proof of the main theorem and some examples} \label{proof}

The goal of this final section is to prove Theorem \ref{intro}, thus establishing that an integral Springer theorem holds for embeddings of quadratic lattices over $R$ under suitable assumptions. We shall also give examples to show that the theorem can fail without our extra assumptions.
%
%In this section, $R$ denotes a general Dedekind domain, with $2\in R^\times$ unless otherwise stated. Let $M$ be a quadratic $R$-lattice and $V=F\otimes_RM$. Notations introduced in \S\;\ref{sec4.1} will be used.

As in \S\;\ref{sec4}, $R$ denotes a general Dedekind domain whose fraction field $F$ has characteristic $\mathrm{char}(F)\neq 2$. Let $M$ be a quadratic $R$-lattice and $V=F\otimes_RM$. Notations introduced in \S\;\ref{sec4.1} will be used.

\subsection{Norm principles for  global integral and adelic spinor norms} \label{norm}

In Theorem \ref{relative} we have obtained the norm principle for spinor norms of transporters between two lattices in the local case. Now we prove  a global integral version and an adelic version  of the norm principle. The adelic norm principle will play a key role in the proof of our main theorem. While the global norm principle is not directly related to our main theorem, we discuss it here for the sake of completeness.

Throughout this subsection, let $N$ be a (non-degenerate) sublattice of $M$.

For the discussion of the global version of Theorem \ref{relative}, we need to assume the strong approximation property of the spin group $\mathbf{Spin}(V)$.

\begin{prop} \label{glb-norm} Suppose that the spin group $\mathbf{Spin}(V)$ satisfies strong approximation over $R$.

Then
\[
 \theta (X^+(M/N)) = \{ \alpha \in \theta (\mathrm{O}^+(V)): \ \alpha \in \theta_{\fp} (X^+(M_\mathfrak p/N_{\mathfrak p} )) \ \ \text{for all $\mathfrak p\in \Spm(R)$} \}.
\] In particular, $\theta (X^+(M/N))$ is a subgroup of $F^\times$.
\end{prop}
\begin{proof}
    For any $\alpha \in \theta (\mathrm{O}^+(V))$, there is $\sigma\in \mathrm{O}^+(V)$ such that $\theta(\sigma)=\alpha (F^\times)^2$. On the other hand, if $\alpha$ satisfies the local conditions, then in the same way as in the proof of Theorem \ref{relspinor} one obtains $\underline{\tau}=(\tau_\fp)\in X_{\mathbf A}^+(M/N)$ such that $\theta_{\mathbf A}(\underline{\tau})=\alpha \bfI_F^2$, and so $\sigma^{-1} \underline{\tau} \in \rO'_{\bfA}(V)$.

    By \eqref{4.2.2} (cf. the proof of Lemma \ref{SAcoro}), there exist
    \[
    \gamma\in\ker(\theta:\rO^+(V)\to F^\times/(F^\times)^2)\quad\text{ and }\quad \underline{\delta}\in\rO_{\bfA}^+(M)
    \] such that $\sigma^{-1} \underline{\tau}=\gamma\underline{\delta}$.
Thus, using \eqref{4.2.1} we find that $\sigma \gamma \in X_{\bfA}^+(M/N)\rO_{\bfA}^+(M)=X_{\bfA}^+(M/N)$. Since $\rO^+(V)\cap X_{\bfA}^+(M/N)=X^+(M/N)$, we get $\sigma\gamma\in X^+(M/N)$ and hence
$$ \alpha (F^\times)^2=\theta(\sigma)=\theta(\sigma \gamma) \in \theta (X^+(M/N)).$$This proves the local-global description of  $\theta(X^+(M/N))$.

The group structure of $\theta (X^+(M/N))$ follows from the local-global relation and Proposition \ref{proposition4}.
\end{proof}

\begin{remark}\label{remark5.2}
  In Proposition \ref{glb-norm}, if we further assume that  there exists a reflection of spinor norm $1$ in $\rO(M_\fp)$ for every $\fp$ (e.g. if $M$ is unimodular as an $R$-lattice and $V$ is isotropic), then a similar argument as above shows that
  \begin{equation}\label{5.1.1}
  \theta (X(M/N)) = \{ \alpha \in \theta (\mathrm{O}(V)): \ \alpha \in \theta_{\fp} (X(M_\mathfrak p/N_{\mathfrak p} )) \ \ \text{for all $\mathfrak p\in \Spm(R)$} \},
  \end{equation}and in particular that $\theta(X(M/N))$ is a subgroup of $F^\times$. However, this local-global description
 for $\theta (X(M/N))$ does not hold in general, as the following counterexample indicates.

 Let $R=k[t]$ be a polynomial ring over a field $k$ with $\car(k)\neq2$, and $\fp_1=(p_1(t))$ and $\fp_2=(p_2(t))$ be two different places of $R$. Take $N=M$ to be the isotropic $R$-lattice $\langle p_1p_2,-p_1p_2,p_1p_2 \rangle$. Then by definition $\theta(X(M/N))=\theta(\rO(M))$ and $\theta(X^+(M/N))=\theta(\rO^+(M))$. Note that for any $\fp\neq\fp_1,\fp_2$, $M_\fp$ is unimodular and so $\theta_\fp(\rO(M_\fp))=\theta_\fp(\rO^+(M_\fp))=R_\fp^\times(F_\fp^\times)^2$ by Proposition \ref{modular-spinor}. On the other hand, $M_{\fp_i}$ is modular of scale $\fp_i$ for both $i=1,2$, so again by Proposition \ref{modular-spinor}, one has $\theta_{\fp_i}(\rO(M_{\fp_i}))=F_{\fp_i}^\times$ while $\theta_{\fp_i}(\rO^+(M_{\fp_i}))=R_{\fp_i}^\times(F_{\fp_i}^\times)^2$. By Proposition \ref{glb-norm}, we have
\[
\theta(\rO^+(M))=A^+\coloneqq\{\alpha\in F^\times: \alpha\in R_\fp^\times(F_\fp^\times)^2 \text{ for all }\fp\in\Spm(R)\}.
\]
Suppose we also have the local-global description
\[
\theta(\rO(M))=A\coloneqq\{\alpha\in F^\times: \alpha\in R_\fp^\times(F_\fp^\times)^2 \text{ for all }\fp\neq\fp_1,\fp_2\}.
\]
Then it would imply $[A:A^+]=[\theta(\rO(M)):\theta(\rO^+(M))]\leq2$. But this is not the case, because $p_1,p_2$ and their quotient $p_1/p_2$ all lie in $A\backslash A^+$. Therefore \eqref{5.1.1} fails in this case. %\qed
\end{remark}

Let $E/F$ be a finite extension and let $S$ be the integral closure of $R$ in $E$. By a well known theorem of Krull and Akizuki, $S$ is a Dedekind domain as $R$ is. For any $R$-lattice, we put $\widetilde{L}=S\otimes_RL$.

An immediate consequence of Proposition \ref{glb-norm} is the following norm principle for $\theta (X^+(M/N))$ (and for $\theta (X(M/N))$ under an extra assumption).

\begin{coro}\label{coro5.2}
With the hypotheses  in Proposition $\ref{glb-norm}$, assume further that $2\in R^\times$.

Then
 $$  N_{E/F}(\theta(X^+(\widetilde{M}/\widetilde{N})))\subseteq\theta(X^+(M/N)). $$

In particular, in the case $M=N$ we obtain
\[
  N_{E/F}\big(\theta(\rO^+(\widetilde{M}))\big)\subseteq \theta(\rO^+(M)) .
\]

When (\ref{5.1.1}) holds for $M$ and $N$, we also have $$N_{E/F}(\theta(X(\widetilde{M}/\widetilde{N})))\subseteq\theta(X(M/N)) \quad\text{and in particular}\quad N_{E/F}\big(\theta(\rO(\widetilde{M}))\big)\subseteq \theta(\rO(M)).$$
\end{coro}
\begin{proof}
    This follows from Proposition \ref{glb-norm}, Proposition \ref{proposition4}, Theorem \ref{relative} and  Knebusch's norm principle \cite[Satz]{K71}.
\end{proof}

The following example explains that the norm principle for integral spinor norms over Dedekind domains does not hold in general.

\begin{example}\label{NoNP}
Let $R$ be the ring of integers in the number field $F= \mathbb Q(\sqrt{3})$. Let $V$ be a quadratic space of dimension $\dim(V)\ge 3$ over $F$ with $\det(V)=1$ such that the real quadratic space $F_v\otimes_FV$ is positive definite for each real place $v$ of $F$. By the construction in \cite[Theorem 3.1]{O'M75}, there is a quadratic lattice $L$ over $R$ with $L\otimes_R F=V$ such that $\mathrm{O}(L)= \{ \pm 1 \}$, that $L_{\mathfrak p}$ is unimodular at every dyadic prime $\mathfrak p$, and that the number of Jordan components of $L_{\mathfrak p}$ is at most 2 over every non-dyadic prime $\mathfrak p$. By \cite[55:7]{O'M00}, one has $\theta (\mathrm{O}^+(L))=(F^\times)^2$.

Let $E=F(\sqrt{-1})$. By strong approximation for spin groups over $E$ (cf. \cite[104:4]{O'M00}), one can use \cite[101:8]{O'M00} and argue similarly as in the proof of Proposition \ref{glb-norm} to get
$$ \theta \big(\mathrm{O}^+(\widetilde{L})\big)= \{ \beta \in E^\times : \ \beta\in \theta_{\mathfrak{q}} \big(\mathrm{O}^+ (\widetilde{L}\otimes_SS_{\mathfrak{q}})\big)  \text{ for all }  \mathfrak{q}\in \Spm(S) \}, $$
where for every $\mathfrak{q}\in\Spm(S)$, $S_{\mathfrak{q}}$ denotes the completion of $S$ at $\mathfrak{q}$. By \cite[92:5]{O'M00} and \cite[Proposition A]{Hsia75}, one obtains
$\theta (\mathrm{O}^+ (L\otimes_R S_{\mathfrak{q}})) \supseteq S_{\mathfrak q}^\times (E_{\mathfrak{q}}^\times)^2 $ for all $\mathfrak{q}\in \Spm(S)$. Since every $\mathfrak p \in \Spm(R)$ is unramified in $E/F$, one has $N_{E_{\mathfrak q}/F_{\mathfrak p}} (S_{\mathfrak q}^\times) = R_{\mathfrak p}^\times $ for all $\mathfrak p\in \Spm(R)$ with $\mathfrak q$ above $\mathfrak p$. This together with the Hasse norm theorem (cf. \cite[65:23]{O'M00}) implies that
$$ N_{E/F}\big(\theta (\mathrm{O}^+(\widetilde{L}))\big) \supseteq  \{ \alpha\in R^{\times}: \alpha >0  \ \text{in $F_v$ for all real places $v$ of $F$} \} .  $$ Therefore the fundamental unit $2-\sqrt{3}$ of $R$ belongs to  $ N_{E/F}\big(\theta (\rO^+(\widetilde{L}))\big)$. Therefore,
 \[
 N_{E/F}\big(\theta (\mathrm{O}^+(\widetilde{L}))\big)\nsubseteq \theta (\mathrm{O}^+(L))=(F^\times)^2
 \]in this example. \qed
\end{example}

In the rest of this subsection, we prove an adelic norm principle.

For each $\mathfrak{q}\in\Spm(S)$, let $S_{\mathfrak{q}}$ denote the completion of $S$ at $\mathfrak{q}$ and let $E_{\mathfrak{q}}$ denote the fraction field of $S_{\mathfrak{q}}$. For each $\fp\in\Spm(R)$ we have the local norm map
\[
N_{E_{(\fp)}/F_{\fp}}: E_{(\fp)}:=\prod_{\fq\,|\,\fp}E_{\fq}\longrightarrow F_{\fp}\,.
\]So there is an induced norm map between id\`ele groups
$N_{E/F}:\bfI_E\to \bfI_F$.
By abuse of notation,  the induced norm map from $\bfI_E/\bfI_E^2$ to $\bfI_F/\bfI_F^2$ will also be denoted by $N_{E/F}$.

\begin{prop}\label{NP-adelic}
With notation and hypotheses as above, assume further that $2\in R^\times$. Then
\[
N_{E/F}\big(\theta_{\bfA}(X_{\bfA}(\widetilde{M}/\widetilde{N}))\big)\subseteq \theta_{\bfA}(X_{\bfA}(M/N))\quad\text{and}\quad N_{E/F}\big(\theta_{\bfA}(X^+_{\bfA}(\widetilde{M}/\widetilde{N}))\big)\subseteq \theta_{\bfA}(X^+_{\bfA}(M/N)).
\]
\end{prop}		
\begin{proof}
We only give a proof of the first inclusion, since the other can be proved similarly. Consider an arbitrary $\underline{\sigma}=(\sigma_{\fq})\in X_{\bfA}(\widetilde{M}/\widetilde{N})$. By definition, we have
\[
N_{E/F}\big(\theta_{\bfA}(\underline{\sigma})\big)=
\bigg(\prod_{\fq\,|\,\fp}N_{E_{\fq}/F_{\fp}}(\theta_{\fq}(\sigma_{\fq}))\bigg)_{\fp}\,.
\]For each $\fp\in\Spm(R)$ and each $\fq\in\Spm(S)$ lying over $\fp$, we have $N_{E_{\fq}/F_{\fp}}(\theta_{\fq}(\sigma_{\fq}))
\in\theta_{\fp}(X(M_{\fp}/N_{\fp}))$ by Theorem \ref{relative}. Since $\theta_{\fp}(X(M_{\fp}/N_{\fp}))$ is a group (Proposition \ref{proposition4}), it follows that
\[
\prod_{\fq\,|\,\fp}N_{E_{\fq}/F_{\fp}}(\theta_{\fq}(\sigma_{\fq})\in\theta_{\fp}(X(M_{\fp}/N_{\fp}))\quad\text{ for every } \fp\in\Spm(R).
\] Therefore, there is a family $\underline{\gamma}=(\gamma_{\fp})\in\prod_{\fp}X(M_{\fp}/N_{\fp})$ such that $N_{E/F}\big(\theta_{\bfA}(\underline{\sigma})\big)=(\theta_{\fp}(\gamma_{\fp}))_{\fp}$.
Finally, since $N_{E/F}(\theta_{\mathbf{A}}(\underline{\sigma}))\in\mathbf{I}_F$, we have $\theta_{\mathfrak{p}}(\gamma_{\mathfrak{p}})\in R_{\mathfrak{p}}^\times(F_{\fp}^\times)^2$ for almost all $\mathfrak{p}$. The same argument as in the proof of Theorem \ref{relspinor} shows that in fact we can choose $\gamma_{\mathfrak{p}}$ to be in $\mathrm{O}(M_{\mathfrak{p}})$ for almost all $\mathfrak{p}$. Then we see that $\underline{\gamma}=(\gamma_{\mathfrak{p}})$ is an element of $X_{\bfA}(M/N)$ with $\theta_{\mathbf{A}}(\underline{\gamma})=N_{E/F}(\theta_{\mathbf{A}}(\underline{\sigma}))$. Since $\underline{\sigma}$ was arbitrary, we are through.
\end{proof}

\begin{remark}\label{norm-compatible}
  If $E/F$ is separable, then $E_{(\fp)}=\prod_{\fq\,|\,\fp}E_{\fq}$ is isomorphic to $E_{\fp}:=F_{\fp}\otimes_FE$ for every $\fp\in\Spm(R)$, and hence the adelic norm map $N_{E/F}:\bfI_E\to \bfI_F$ is compatible with the usual norm map $N_{E/F}: E^\times\to F^\times$.

  If $E/F$ is a purely inseparable extension of prime degree $l\neq 2$, then $E=F(\sqrt[l]{b})$ for some $b\in F^\times\setminus (F^\times)^l$, and hence  $F_{\mathfrak{p}}\otimes_FE\cong F_{\mathfrak{p}}[x]/(x^{l}-b).$
   Since the polynomial $x^{l}-b$ either stays irreducible or factors as $(x-c)^{l}$ for some $c\in F_{\fp}$ over $F_{\mathfrak{p}}$, the Jacobson radical of $F_{\mathfrak{p}}\otimes_FE$ is either trivial or generated by the canonical image of $x-c$. It then follows from \cite[Chapter VI, \S\;8.2, Proposition 2 (b)]{Bourbaki} that there is a unique  $\mathfrak{q}\in\Spm(S)$ lying above $\mathfrak{p}$ and that
   \[
 [E_{(\fp)}:F_{\fp}]=[E_{\mathfrak{q}}:F_{\mathfrak{p}}]=
\begin{cases}
l \quad &\text{ if }  x^{l}-b \text{ is irreducible over }F_{\mathfrak{p}}\,,\\
1 \quad &\text{ otherwise}\,.
\end{cases}\]Note that the norm map $N_{K/k}: K\to k$ for any purely inseparable degree $l$ extension $K/k$ is given by
$N_{K/k}(\alpha)=\alpha^l$ for all $\alpha\in K$. So the two norm maps
\[
N_{E/F}:E^\times\to F^\times\quad\text{ and }\quad N_{E_{(\fp)/F_{\fp}}}: E^\times_{(\fp)}\to F_{\fp}^\times
\]are clearly compatible if $[E_{(\fp)}:F_{\fp}]=l$. Since $l$ is odd, these two norms are compatible modulo squares  if $[E_{(\fp)}:F_{\fp}]=1$.

The above discussions together show that the diagram
\[
\begin{CD}
  E^\times @>>> \bfI_E/\bfI_E^2\\
  @V{N_{E/F}}VV @VV{N_{E/F}}V\\
  F^\times @>>> \bfI_F/\bfI_F^2
\end{CD}
\]
is commutative in the general case (where $E/F$ need not be separable or purely inseparable). Therefore, the norm principles in Corollary \ref{coro5.2} and Proposition \ref{NP-adelic} are compatible (when they hold).  \qed
\end{remark}

\subsection{Proof of the main theorem}\label{sec5.2}

The goal of this subsection is to give a proof of Theorem \ref{intro} based on results obtained previously.

Let $E/F$ be a finite extension and let $S$ be the integral closure of $R$ in $E$.

\begin{lemma}\label{odd-deg}
Suppose that $[E:F]$ is odd. Then, for each $\fp\in\Spm(R)$, there exists a $\mathfrak{q}\in\Spm(S)$ lying over $\fp$ such that $[E_{\fq}:F_{\fp}]$ is odd.
\end{lemma}
\begin{proof}
  When $E/F$ is separable, this follows from the fundamental equality $ [E:F]=\sum_{\fq|\mathfrak{p}}[E_{\fq}:F_{\mathfrak{p}}].$  The case with $E/F$ purely inseparable of prime degree has been discussed in Remark \ref{norm-compatible}. The general case follows easily from a combination of the above two special cases.
\end{proof}

Let $N$ be a quadratic lattice in $V=F\otimes_{R}M$. For any $R$-lattice $L$ on $V$, (as in \S\;\ref{sec4.2}) we say that $N$ embeds into $\spn(L)$ (resp. $\spn^+(L)$) and we write $N\hookrightarrow \spn(L)$ (resp. $N\hookrightarrow\spn^+(L)$) if $N$ is contained in some lattice in $\spn(L)$ (resp. $\spn^+(L)$).

We have the following Springer theorem for the embedability into a spinor genus.

\begin{thm}\label{spin-Springer}
 Suppose that $[E:F]$ is odd and $2\in R^\times$. Let $N$ be a quadratic lattice in $V$.

If $\widetilde{N}\hookrightarrow\spn^+(\widetilde{M})$ over $S$, then $N\hookrightarrow\spn^+(M)$ over $R$.

Similarly, if $\widetilde{N}\hookrightarrow\spn(\widetilde{M})$ over $S$, then $N\hookrightarrow\spn(M)$ over $R$.
\end{thm}
\begin{proof}
  First assume that $\widetilde{N}\hookrightarrow\spn^+(\widetilde{M})$. Then by definition, there is an $S$-lattice $L\in\spn^+(\widetilde{M})\subseteq \gen(\widetilde{M})$ such that $\widetilde{N}$ is contained in $L$. Thus, for every $\fq\in\Spm(S)$, $S_{\fq}\otimes_S\widetilde{N}=S_{\fq}\otimes_RN$ is contained in $S_{\fq}\otimes_SL\cong S_{\fq}\otimes_S\widetilde{M}=S_{\fq}\otimes_RM$. For every $\fp\in \Spm(R)$, we can deduce from Lemma \ref{odd-deg} and Theorem \ref{local-springer} that $N_{\fp}=R_{\fp}\otimes_RN$ can be embedded in $M_{\fp}=R_{\fp}\otimes_RM$ over $R_{\fp}$. So, by \cite[102:5]{O'M00}, $N$ is embedded in some lattice $P\in \gen(M)$. Without loss of generality, we may assume $N\subseteq P$.

  As we have mentioned in \S\;\ref{sec4.2} (as a consequence of Theorem \ref{relspinor}), if the set of all proper spinor genera in $\gen(M)=\gen(P)$ is identified with the quotient group
  $\rO_{\bfA}^+(V)\big/\big(\rO^+(V)\rO'_{\bfA}(V)\rO_{\bfA}^+(M)\big)$ as in \eqref{spin-genus}, then the subset consisting of proper spinor genera in which $N$ embeds can be characterized as the normal subgroup
  \[
  \rO^+(V)\rO'_{\bfA}(V)X_{\bfA}^+(P/N)\big/\big(\rO^+(V)\rO'_{\bfA}(V)\rO_{\bfA}^+(M)\big) \lhd \rO_{\bfA}^+(V)\big/\big(\rO^+(V)\rO'_{\bfA}(V)\rO_{\bfA}^+(M)\big).
  \]Let $[\spn^+(M)]$ be the canonical image of $\spn^+(M)$ in the quotient
  \[
  \{\spn^+(L)\,|\,L\in\gen(P)\}\big/\{\spn^+(L)\,|\,L\in\gen(P),\,N\hookrightarrow \spn^+(L)\}\cong \rO_{\bfA}^+(V)\big/\rO^+(V)\rO'_{\bfA}(V)X_{\bfA}^+(P/N).
  \]We need to show that $[\spn^+(M)]=0$ in the above quotient group.

  Put $V_E=E\otimes_FV$ and consider the following commutative diagram
  $$
		\xymatrix{
			& \mathrm{O}_{\mathbf{A}}^+(V)\big/\big(\rO^+(V)\rO'_{\bfA}(V)X_{\mathbf A}^+(P/N)\big)  \ar[r]^{\theta_F} \ar[d]^\phi & \mathbf{I}_F/\theta(\mathrm{O}^+(V))\theta_{\mathbf{A}}(X_{\mathbf A}^+(P/N)) \ar[d]^\psi\\
			& \mathrm{O}_{\mathbf{A}}^+(V_E)\big/\big(\mathrm{O}^+(V_E)\rO'_{\bfA}(V_E)X_{\mathbf A}^+(\widetilde{P}/\widetilde{N})\big)\ar[r]^{} & \mathbf{I}_E/\theta(\mathrm{O}^+(V_E))\theta_{\mathbf{A}}(X_{\mathbf A}^+(\widetilde{P}/\widetilde{N}))
		}
		$$ where the horizontal maps are induced by the adelic spinor norm maps and the vertical maps are induced by the natural inclusions
$\rO^+(V)\hookrightarrow\rO^+(V_E)$ and $\bfI_F\hookrightarrow\bfI_E$. It is clear from the definition that the map $\theta_F$ in the diagram is injective.
The assumption $\widetilde{N}\hookrightarrow\spn^+(\widetilde{M})$ means that $[\spn^+(M)]$ lies in the kernel of $\phi$. So it remains to prove the injectivity of the map $\psi$ in the above diagram.

By Knebusch's norm principle, the adelic norm principle (Proposition \ref{NP-adelic}) and their compatibility (as discussed in Remark \ref{norm-compatible}), there is a well-defined norm map
\[
N_{E/F}:  \  \mathbf{I}_E/\theta(\mathrm{O}^+(V_E))\theta_{\mathbf{A}}(X_{\mathbf A}^+(\widetilde{P}/\widetilde{N}))\longrightarrow\mathbf{I}_F/\theta(\mathrm{O}^+(V))\theta_{\mathbf{A}}(X_{\mathbf A}^+(P/N)).
\]
Since $[E:F]$ is odd and the groups under consideration are abelian 2-torsion groups, one concludes that $N_{E/F}\circ\psi=\mathrm{Id}$. This implies that $\psi$ is injective. The first assertion is thus proved.

To prove the second assertion, assume that $\widetilde{N}\hookrightarrow\spn(\widetilde{M})$. In other words, $\widetilde{N}\subseteq\sigma\underline{\rho}(\widetilde{M})$ for some $\sigma\in\rO(V_E)$ and $\underline{\rho}\in \rO_{\bfA}'(V_E)$. We may assume $\det(\sigma)=-1$. In this case choose any $\gamma\in\rO(V)$ with $\det(\gamma)=-1$, then $$\widetilde{\gamma N}=\gamma\widetilde{N}\subseteq\gamma\sigma\underline{\rho}(\widetilde{M}) \quad\text{with}\quad \gamma\sigma\underline{\rho}(\widetilde{M})\in \spn^+(\widetilde{M}).$$ Therefore $\gamma N\subseteq M'$ for some lattice $M'\in\spn^+(M)$ by the previous part, and so $N$ is contained in the lattice $\gamma^{-1}M'\in\spn(M)$.
\end{proof}

%\begin{remark}\label{rmk-localisotropy}
  %Suppose that there is an integer $u(R)\in\mathbb{N}$ such that for almost all $\fp\in\Spm(R)$, quadratic forms of dimension $>u(R)$ over the residue field $R/\fp$ are all isotropic. Then, by Hensel's lemma, $V_{\fp}$ is isotropic for almost all $\fp\in\Spm(R)$ when $\dim V>u(R)$, since $M_{\fp}$ is unimodular for almost all $\fp\in\Spm(R)$.
%\end{remark}
%
%\begin{example}
%  Let $R=k[C]$ be the coordinate ring of a normal irreducible affine curve $C$ defined over a field $k$ with $\car(k)\neq2$. Then for any odd degree extension $E/F$ of the fraction field $F$ and any quadratic $R$-lattices $N,\,M$, we have $N\hookrightarrow\spn(M)$ over $R$ if $S\otimes_RN\hookrightarrow \spn(S\otimes_RM)$ over $S$, the integral closure of $R$ in $E$.
%\end{example}

\begin{coro}\label{springer-coro}
With notation and hypotheses as in Theorem $\ref{spin-Springer}$, assume further that $\mathbf{Spin}(V)$ satisfies strong approximation over $R$.

If $\widetilde{N}$ is contained in some lattice in $\cls(\widetilde{M})$ (resp. in $\cls^+(\widetilde{M})$) over $S$, then $N$ is contained in some lattice in $\cls(M)$ (resp. in $\cls^+(M)$) over $R$.
\end{coro}
\begin{proof}
  Combine Theorem \ref{spin-Springer} with Lemma \ref{SAcoro}.
\end{proof}

Now a proof of our main theorem is immediate.

\begin{proof}[Proof of Theorem  \ref{intro}]
Notice that $N$ is not a priori a lattice in $V=F\otimes_{R}M$. But since $S\otimes_{R}N$ can be embedded into $S\otimes_{R}M$ over $S$, we also have an embedding of quadratic spaces $E\otimes_{R}N\rightarrow E\otimes_{R}M$. So the classical Springer theorem for quadratic spaces ensures that the space $F\otimes_RN$ can be embedded in $V=F\otimes_RM$, and thus we may view $N$ as a lattice in $V$ without loss of generality.

At this point the theorem follows from a combination of Corollary \ref{springer-coro} with Remark \ref{rmkSA}, keeping in mind that for lattices $N$ and $M$ in a common quadratic space, $N$ can be embedded into $M$ if and only if $N$ is contained in some lattice in $\cls(M)$.
\end{proof}

\subsection{Examples}\label{sec5.3}

\begin{example} \label{c-exam}
The result of Theorem \ref{intro} can fail when $V=F\otimes_RM$ is 2-dimensional. To see this, consider the quadratic extension of $\mathbb {F}_3(t)$ given by $F= \mathbb{F}_3(t)(\sqrt{t^3-t^2+1})$  and let $R$ be the integral closure of $\mathbb F_3[t]$ inside $F$. Since the class number of $R$ is 5 by the table given in \cite[p.232]{Artin24}, there is a non-principal ideal $\mathfrak a$ of $R$ such that $\mathfrak{a}^5=rR$ for some $r\in R$ with $\sqrt[5]{r}\not\in F$. The quadratic lattice $M$ over $R$ defined by
$$ M= \mathfrak{a}x+\mathfrak{a}^{-1}y \ \ \text{with}\ \ Q(x)=Q(y)=0 \ \ \text{and} \ \ \langle x,y\rangle =\frac{1}{2} $$
 satisfies $1 \not \in Q(M)$. Indeed, if $1\in Q(M)$, there would exist $a\in\mathfrak{a}$ and $b\in\mathfrak{a}^{-1}$ such that $1=ab$, implying that $$\mathfrak{a}=ab\mathfrak{a}\subseteq aR\subseteq \mathfrak a .$$ Thus, $\mathfrak a=aR$ would be a principal ideal, contradictory to the choice.

 Let $S$ be the integral closure of $R$ in the degree 5 extension $E=F(\sqrt[5]{r})$. Then $\mathfrak{a}S=\sqrt[5]{r} S$,
\[
\widetilde{M}=S\otimes_RM= S \sqrt[5]{r} x + S \sqrt[5]{r}^{-1} y\quad\text{ and }\quad 1=Q(\sqrt[5]{r} x +  \sqrt[5]{r}^{-1} y) \in Q(\widetilde{M}).
\]This shows that the 1-dimensional diagonal lattice $N:=\langle 1\rangle$ cannot be embedded in $M$ over $R$, but $\widetilde{N}=S\otimes_RN$ can be embedded in $\widetilde{M}=S\otimes_RM$ over $S$.
\end{example}

\begin{example}\label{isotropy-exam}
  This example is to show that Theorem \ref{intro} can be false if $V=F\otimes_RM$ is anisotropic.
  Let $R=\mathbb{F}_q[t]$ be a one-variable polynomial ring over a finite field $\mathbb{F}_q$ of odd characteristic. Let $\delta\in\mathbb{F}_q^*$ be a non-square.
 By \cite[Lemma\;4.1]{HuLiuTian23},  the element $t\in R$ is not represented by the diagonal lattice $M:=\langle 1,-\delta,t\delta,-t^3\rangle$ over $R$. In other words, the lattice $N:=\langle t\rangle$ cannot be embedded into $M$ over $R$.

  Set $s=t^{-1}\in F=\mathbb{F}_q(t)$ and $E=F(\theta)$, where $\theta$ is a root of $f(X):=X^3-\delta X-s$. We claim that over the integral closure $S$ of $R$ in $E$, the lattice
  $S\otimes_RN$ can be embedded into $S\otimes_RM$.

  Indeed, the completion of $F$ at the infinite place (which corresponds to the $s$-adic valuation) is $F_{\infty}=\mathbb{F}_q(\!(s)\!)$. By Hensel's lemma, the polynomial $f$ admits a factorization $f=gh$, where $g,\,h$ are monic polynomials with coefficients in $R'=\mathbb{F}_q[\![s]\!]$ such that
  \[
  g\equiv X\pmod{s}\quad\text{and}\quad h\equiv X^2-\delta\pmod{s}.
  \]Hence, $E$ has two places $\infty_1,\,\infty_2$ lying over the infinite place of $F$, with completions
  $E_{\infty_1}\cong F_{\infty}$ and $E_{\infty_2}\cong F_{\infty}[X]/(h)$. Since $\delta$ is a square in the residue field of $E_{\infty_2}$, the quadratic space $E_{\infty_2}\otimes_RM$ is isotropic. By \cite[Thm.\;3.5 (ii)]{HuLiuTian23} (or \cite[p.135, lines 2--4]{HSX98}), the representability of $S$-lattices by $S\otimes_RM$ satisfies the Hasse principle. So it suffices to prove that $S\otimes_RM$ represents $S\otimes_RN$ everywhere locally. But this follows immediately from the fact that $N$ is locally represented by $M$ over completions of $R$ (\cite[Thm.\;4.2 (2)]{HuLiuTian23}). Our claim is thus proved.
\end{example}

\begin{example}\label{NumberField-exam}
  This example is to show that Question \ref{Spring-question} can have a negative answer if $R$ is the ring of integers in a number field $F$ and $F\otimes_RM$ is definite  (i.e., anisotropic over every archimedean completion of $F$). We take $F=\mathbb{Q}(\sqrt{13})$ and let $K=\mathbb{Q}(\theta)$, where $\theta$ is a root of $X^3+X^2-1$. Let $E=FK=\mathbb{Q}(\sqrt{13},\theta)$.

  Let $R$ (resp. $S$) be the ring of integers of $F$ (resp. $E$). The element
  \[
  \alpha=3+\bigg(\frac{1+\sqrt{13}}{2}\bigg)^2+\bigg(1+\frac{1+\sqrt{13}}{2}\bigg)^2
  \]is not a sum of 4 squares in $R$ (cf. \cite[Thm.\;3.1]{KrasenskyRavkaSgallova22}). In other words, the lattice $N:=\langle \alpha\rangle$ is not represented by the diagonal lattice $M:=\langle 1,1,1,1\rangle$ over $R$. We claim that $S\otimes_RN$ is represented by $S\otimes_RM$ over $S$.

  Indeed, the cubic field $K$ is not totally real, so the quadratic space $E\otimes_RM$ is indefinite. By \cite[p.135, lines 2--4]{HSX98} (or the proof of \cite[Thm.\;3.5 (ii)]{HuLiuTian23}), the representability of $S$-lattices by $S\otimes_RM$ satisfies the Hasse principle. So it suffices to prove that $S\otimes_RM$ represents $S\otimes_RN$ everywhere locally, i.e., $\alpha$ is a sum of 4 squares in the local completions of $S$ at all places of $E$. At a non-dyadic place the result follows from \cite[Prop.\;2.3]{XuZhang22TAMS}. Note that $E$ has discriminant $d_E=23^2\cdot 13^3$, by \cite[Prop.\;I.2.11]{Neu2}. So $2$ is unramified in $E$. Therefore, at each dyadic place of $S$ we can use \cite[Prop.\;4.3]{HeHu23JPAA} to deduce that $\alpha$ is represented by $S\otimes_RM$ over that completion. This proves our claim.
\end{example}

\begin{remark}\label{definite-exam}
  More examples of integral quadratic forms over number fields can be found in \cite{DaansKalaKrasenskyYatsyna24} answering Question \ref{Spring-question} in the negative, where as opposed to Example \ref{isotropy-exam} or \ref{NumberField-exam}, even the quadratic space $E\otimes_RM$ is definite.
\end{remark}

\subsection*{Acknowledgements}The first author was supported by the National Natural Science Foundation of China (grant no.\,12171223). The second author was partially supported by the short-term visiting program for doctoral students at Southern University of Science and Technology. The third author was supported by the National Key R\&D Program of China No.\,2023YFA1009702 and the National Natural Science Foundation of China (grant no.\,11631009).

\addcontentsline{toc}{section}{\textbf{References}}

	\bibliographystyle{siam}
	\bibliography{ref}

\begin{thebibliography}{10}

\bibitem{Artin24}
{\sc E.~Artin}, {\em Quadratische {K}\"orper im {G}ebiete der h\"oheren
  {K}ongruenzen. {II}.}, Math. Z., 19 (1924), pp.~207--246.

\bibitem{BM02}
{\sc P.~Barquero and A.~Merkurjev}, {\em Norm principle for reductive algebraic
  groups}, in Algebra, arithmetic and geometry, {P}art {I}, {II} ({M}umbai,
  2000), vol.~16 of Tata Inst. Fund. Res. Stud. Math., Tata Inst. Fund. Res.,
  Bombay, 2002, pp.~123--137.

\bibitem{BCM19}
{\sc N.~Bhaskhar, V.~Chernousov, and A.~Merkurjev}, {\em The norm principle for
  type {$D_n$} groups over complete discretely valued fields}, Trans. Amer.
  Math. Soc., 372 (2019), pp.~97--117.

\bibitem{Bourbaki}
{\sc N.~Bourbaki}, {\em Commutative algebra. {C}hapters 1--7}, Elements of
  Mathematics (Berlin), Springer-Verlag, Berlin, 1998.
\newblock Translated from the French, Reprint of the 1989 English translation.

\bibitem{CT}
{\sc J.-L. Colliot-Th\'el\`ene}, {\em Approximation forte pour les espaces
  homog\`enes de groupes semi-simples sur le corps des fonctions d'une courbe
  alg\'ebrique complexe}, Eur. J. Math., 4 (2018), pp.~177--184.

\bibitem{DaansKalaKrasenskyYatsyna24}
{\sc N.~Daans, V.~Kala, J.~Kr\'asensk\'y, and P.~Yatsyna}, {\em Failures of
  integral {S}pringer's theorem}, arXiv:2404.12844,  (2024).
\newblock \url{https://arxiv.org/abs/2404.12844}.

\bibitem{EH77}
{\sc A.~G. Earnest and J.~S. Hsia}, {\em Spinor genera under field extensions.
  {I}}, Acta Arith., 32 (1977), pp.~115--128.

\bibitem{EH78}
\leavevmode\vrule height 2pt depth -1.6pt width 23pt, {\em Spinor genera under
  field extensions. {II}. {$2$} unramified in the bottom field}, Amer. J.
  Math., 100 (1978), pp.~523--538.

\bibitem{G97}
{\sc P.~Gille}, {\em La {$R$}-\'{e}quivalence sur les groupes alg\'{e}briques
  r\'{e}ductifs d\'{e}finis sur un corps global}, Inst. Hautes \'{E}tudes Sci.
  Publ. Math.,  (1997), pp.~199--235 (1998).

\bibitem{G09}
\leavevmode\vrule height 2pt depth -1.6pt width 23pt, {\em Le probl\`eme de
  {K}neser-{T}its}, in S\'{e}minaire Bourbaki. Vol. 2007/2008. Exp. 983, Paris:
  Soci\'et\'e Math\'ematique de France (SMF), 2009, pp.~39--81.

\bibitem{GN21}
{\sc P.~Gille and E.~Neher}, {\em Springer's odd degree extension theorem for
  quadratic forms over semilocal rings}, Indag. Math. (N.S.), 32 (2021),
  pp.~1290--1310.

\bibitem{Harder67}
{\sc G.~Harder}, {\em Halbeinfache {G}ruppenschemata \"uber {D}edekindringen},
  Invent. Math., 4 (1967), pp.~165--191.

\bibitem{H23}
{\sc Z.~He}, {\em Arithmetic {S}pringer theorem and $n$-universality under
  field extensions}, arXiv:2312.09560,  (2023).
\newblock \url{https://arxiv.org/abs/2312.09560}.

\bibitem{HeHu23JPAA}
{\sc Z.~He and Y.~Hu}, {\em On {$k$}-universal quadratic lattices over
  unramified dyadic local fields}, J. Pure Appl. Algebra, 227 (2023), pp.~Paper
  No. 107334, 32.

\bibitem{Hsia75}
{\sc J.~S. Hsia}, {\em Spinor norms of local integral rotations. {I}}, Pacific
  J. Math., 57 (1975), pp.~199--206.

\bibitem{HSX98}
{\sc J.~S. Hsia, Y.~Y. Shao, and F.~Xu}, {\em Representations of indefinite
  quadratic forms}, J. Reine Angew. Math., 494 (1998), pp.~129--140.

\bibitem{HuLiuTian23}
{\sc Y.~Hu, J.~Liu, and Y.~Tian}, {\em Strong approximation and {H}asse
  principle for integral quadratic forms over affine curves}.
\newblock to appear in Acta Arithmetic, available at
  \url{https://arxiv.org/abs/arXiv:2312.08849}, 2023.

\bibitem{K71}
{\sc M.~Knebusch}, {\em Ein {S}atz \"{u}ber die {W}erte von quadratischen
  {F}ormen \"{u}ber {K}\"{o}rpern}, Invent. Math., 12 (1971), pp.~300--303.

\bibitem{Kn56}
{\sc M.~Kneser}, {\em Klassenzahlen indefiniter quadratischer {F}ormen in drei
  oder mehr {V}er\"{a}nderlichen}, Arch. Math. (Basel), 7 (1956), pp.~323--332.

\bibitem{KMRT98}
{\sc M.-A. Knus, A.~Merkurjev, M.~Rost, and J.-P. Tignol}, {\em The book of
  involutions}, vol.~44 of American Mathematical Society Colloquium
  Publications, American Mathematical Society, Providence, RI, 1998.
\newblock With a preface in French by J. Tits.

\bibitem{KrasenskyRavkaSgallova22}
{\sc J.~Kr\'{a}sensk\'{y}, M.~Ra\v{s}ka, and E.~Sgallov\'{a}}, {\em Pythagoras
  numbers of orders in biquadratic fields}, Expo. Math., 40 (2022),
  pp.~1181--1228.

\bibitem{Lam05}
{\sc T.~Y. Lam}, {\em Introduction to quadratic forms over fields}, vol.~67 of
  Graduate Studies in Mathematics, American Mathematical Society, Providence,
  RI, 2005.

\bibitem{M95}
{\sc A.~Merkurjev}, {\em The norm principle for algebraic groups}, Algebra i
  Analiz, 7 (1995), pp.~77--105.

\bibitem{Neu2}
{\sc J.~Neukirch}, {\em Algebraic number theory}, vol.~322 of Grundlehren der
  Mathematischen Wissenschaften [Fundamental Principles of Mathematical
  Sciences], Springer-Verlag, Berlin, 1999.
\newblock Translated from the 1992 German original and with a note by Norbert
  Schappacher, With a foreword by G. Harder.

\bibitem{OPZ04}
{\sc M.~Ojanguren, I.~Panin, and K.~Zainoulline}, {\em On the norm principle
  for quadratic forms}, J. Ramanujan Math. Soc., 19 (2004), pp.~289--300.

\bibitem{O'M58}
{\sc O.~T. O'Meara}, {\em The integral representations of quadratic forms over
  local fields}, Amer. J. Math., 80 (1958), pp.~843--878.

\bibitem{O'M75}
\leavevmode\vrule height 2pt depth -1.6pt width 23pt, {\em The construction of
  indecomposable positive definite quadratic forms}, J. Reine Angew. Math., 276
  (1975), pp.~99--123.

\bibitem{O'M00}
\leavevmode\vrule height 2pt depth -1.6pt width 23pt, {\em Introduction to
  quadratic forms}, Classics in Mathematics, Springer-Verlag, Berlin, 2000.
\newblock Reprint of the 1973 edition.

\bibitem{PP10}
{\sc I.~Panin and K.~Pimenov}, {\em Rationally isotropic quadratic spaces are
  locally isotropic: {II}}, Doc. Math.,  (2010), pp.~515--523.

\bibitem{PR07}
{\sc I.~Panin and U.~Rehmann}, {\em A variant of a theorem by {S}pringer},
  Algebra i Analiz, 19 (2007), pp.~117--125.

\bibitem{P79}
{\sc V.~P. Platonov}, {\em On the problem of the rationality of spin
  varieties}, Dokl. Akad. Nauk SSSR, 248 (1979), pp.~524--527.

\bibitem{S69}
{\sc W.~Scharlau}, {\em Quadratic forms}, Queen's Papers in Pure and Applied
  Mathematics, No. 22, Queen's University, Kingston, ON, 1969.

\bibitem{Serre-GTM67}
{\sc J.-P. Serre}, {\em Local fields}, vol.~67 of Graduate Texts in
  Mathematics, Springer-Verlag, New York-Berlin, 1979.
\newblock Translated from the French by Marvin Jay Greenberg.

\bibitem{Spr52}
{\sc T.~A. Springer}, {\em Sur les formes quadratiques d'indice z\'{e}ro}, C.
  R. Acad. Sci. Paris, 234 (1952), pp.~1517--1519.

\bibitem{Wit37}
{\sc E.~Witt}, {\em Theorie der quadratischen {F}ormen in beliebigen
  {K}\"{o}rpern}, J. Reine Angew. Math., 176 (1937), pp.~31--44.

\bibitem{X99}
{\sc F.~Xu}, {\em Arithmetic {S}pringer theorem on quadratic forms under field
  extensions of odd degree}, in Integral quadratic forms and lattices ({S}eoul,
  1998), vol.~249 of Contemp. Math., Amer. Math. Soc., Providence, RI, 1999,
  pp.~175--197.

\bibitem{XuZhang22TAMS}
{\sc F.~Xu and Y.~Zhang}, {\em On indefinite and potentially universal
  quadratic forms over number fields}, Trans. Amer. Math. Soc., 375 (2022),
  pp.~2459--2480.

\end{thebibliography}
	
\end{document}